\documentclass[a4paper,11pt,reqno]{amsart}
\usepackage[utf8]{inputenc}
\usepackage{mathtools}
\usepackage{amssymb}
\usepackage{amsfonts}
\usepackage{amsthm}
\usepackage{bbm}
\usepackage{fullpage}
\usepackage{color}
\usepackage{hyperref}
\usepackage{stmaryrd}
\usepackage{enumitem}
\usepackage{accents}
\mathtoolsset{showonlyrefs}
\newcommand{\N}{\mathbb N}

\newcommand{\eps}{\varepsilon}
\renewcommand{\P}{\mathbb P}
\newcommand{\E}{\mathbb E}

\newcommand{\F}{\mathcal F}

\newcommand{\ind}{\mathbbm 1}

\newtheorem{theorem}{Theorem}[section]
\newtheorem{lemma}[theorem]{Lemma}
\newtheorem{proposition}[theorem]{Proposition}
\newtheorem{corollary}[theorem]{Corollary}

\newtheorem{remark}[theorem]{Remark}

\theoremstyle{definition}

\newcommand{\dd}{\text{d}}

\newcommand{\p}{{p^*}}
\newcommand{\f}{\mathfrak{f}}

\newcounter{constants}

\renewcommand{\tilde}{\widetilde}

\newcommand{\bbE}{{\ensuremath{\mathbb E}} }

\newcommand{\bbN}{{\ensuremath{\mathbb N}} }

\newcommand{\bbP}{{\ensuremath{\mathbb P}} }

\newcommand{\bbR}{{\ensuremath{\mathbb R}} }

\newcommand{\bbZ}{{\ensuremath{\mathbb Z}} }

\renewcommand{\epsilon}{\varepsilon}

\newcommand{\gb}{\beta}

\newcommand{\gep}{\varepsilon}       

\newcommand{\gG}{\Gamma}

\newcommand{\go}{\omega}

\newcommand{\gl}{\lambda}

\newcommand{\cG}{{\ensuremath{\mathcal G}} }

\newcommand{\cF}{{\ensuremath{\mathcal F}} }

\newcommand{\cT}{{\ensuremath{\mathcal T}} }

\newcommand{\cZ}{{\ensuremath{\mathcal Z}} }

\newcommand{\bP}{{\ensuremath{\mathbf P}} }

\newcommand{\bE}{{\ensuremath{\mathbf E}} }

\newcommand{\lint}{\llbracket}
\newcommand{\rint}{\rrbracket}
\newcommand{\intp}[1]{\lint #1 \rint}

\newcommand{\maxtwo}[2]{\max_{\substack{#1 \\ #2}}} 

\definecolor{darkred}{rgb}{0.7,0.1,0.1}
\renewcommand{\hat}{\widehat}

\renewcommand{\complement}{\mathsf{c}}

\begin{document}

\author{Stefan Junk}
\address{Gakushuin University, 1-5-1 Mejiro, Toshima-ku, Tokyo 171-8588 Japan}
\email{sjunk@math.gakushuin.ac.jp}
\author{Hubert Lacoin}
\address{IMPA, Estrada Dona Castorina 110, 
Rio de Janeiro RJ-22460-320- Brasil}
\email{lacoin@impa.br}

\title{Coincidence of critical points for directed polymers for general environments and random walks}

\begin{abstract}
For the directed polymer in a random environment (DPRE), two critical inverse-temperatures can be defined.
The first one, $\beta_c$, separates the strong disorder regime (in which the normalized partition function $W^{\beta}_n$ tends to zero) from the weak disorder regime (in which $W^{\beta}_n$ converges to a nontrivial limit). The other, $\bar \beta_c$, delimits the very strong disorder regime (in which $W^{\beta}_n$ converges to zero exponentially fast). It was proved in \cite{JL24} that $\beta_c=\bar \beta_c$ when the random environment is upper-bounded for the DPRE based on the simple random walk. We extend this result to general environment and arbitrary reference walk.
We also prove that $\beta_c=0$ if and only the $L^2$-critical point is trivial.
   \\[10pt]
\end{abstract}

\maketitle

\section{Introduction}
We start by defining the directed polymer in random environment with general reference random walk. For the simple random walk case, an overview of known results can be found in the lecture notes \cite{C17} and the more recent survey \cite{Zy24}.
We let $X=(X_k)_{k\ge 0}$ denote a random walk on $\bbZ^d$ starting from the origin and with independent identically distributed (i.i.d) increments. The associated probability is denoted by $P$.
We let $\eta\in [0,\infty]$ denote the exponent associated with the tail-decay of $|X_1|$, defined as
\begin{equation}\label{taileta}
 -\limsup_{u\to \infty} \frac{\log P(|X_1|\ge u)}{\log u}\eqqcolon\eta\  \in[0,\infty],
\end{equation}
where here and in what follows $|\cdot|$ stands for the Euclidean norm in $\bbR^d$.
The main results in this paper are proved under the assumption that $\eta>0$.

\smallskip Additionally, we consider a collection $\go=(\go_{k,x})_{k\ge 1, x\in \bbZ^d}$ of i.i.d.\ real-valued random variables (we let $\bbP$ denote the associated probability) and make the assumption (throughout the whole paper except for Theorem \ref{strongbeta}) that  $\go_{k,x}$ has exponential moments of all order, that is
\begin{equation}\label{allorder}
	\forall \beta\in \bbR, \ \gl(\beta)\coloneqq\log\  \bbE\left[e^{\beta \go_{1,0}}\right] <\infty.
\end{equation}
The above assumption implies that $\go_{1,0}$ has finite mean and variance, and without loss of generality 
we may assume that
\begin{equation}\label{standard}
\bbE[\go_{1,0}]=0 \quad \text{ and } \quad \bbE[\go_{1,0}^2]=1. 
\end{equation}
Given a realization of $\go$ (the random environment), $\beta\ge 0$ (the inverse temperature) and  $n\ge 1$ (the polymer length), we define the polymer measure $P^\beta_{\go,n}$
as a modification of the distribution $P$ which favors trajectories that visit sites where $\go$ is large, namely
\begin{align}\label{eq:mu}
	P_{\omega,n}^\beta(\dd X)\coloneqq \frac{1}{W^{\beta}_n}e^{\sum^n_{k=1}(\beta \go_{k,X_k}-\gl(\beta))}P(\dd X)  \quad  \text{ where } \quad W^{\beta}_n\coloneqq  E\left[  e^{\sum^n_{k=1}(\beta \go_{k,X_k}-\gl(\beta))}\right].
\end{align}
The quantity $W_{n}^\beta$ (the dependence in $\go$ is omitted in the notation for better readability) is referred to as the (normalized) partition function of the model. A direct application of Fubini yields $\bbE[W_{n}^\beta]=1.$
Using Fubini for the conditional expectation,  it was observed in \cite{B89} that
the process $(W_n^{\beta})_{n\ge 0}$ is a martingale with respect to the filtration $(\cF_n)$ defined by 
\begin{equation}
 \cF_n\coloneqq \sigma\left( \go_{k,x} \colon k\le n, x\in \bbZ^d\right).
\end{equation}
As a consequence, $(W^{\beta}_n)_{n\in\bbN}$ converges almost surely as $n\to\infty$ and we let $W^{\beta}_{\infty}$ denote its limit.
By Kolmogorov's $0$-$1$ law, we have $\bbP(W^{\beta}_{\infty}=0)\in \{0,1\}$.  Using terminology established in \cite{CSY03}, we say that \textit{strong disorder} holds if $\bbP(W^{\beta}_{\infty}=0)=1$ and that \textit{weak disorder} holds if   $\bbP(W^{\beta}_{\infty}>0)=1$.
Finally (relying on \cite[Proposition 2.5]{CSY03} for the existence of the limit, which easily generalizes to the case of general random walks) we define the free energy of the directed polymer by setting
\begin{equation}
 \f(\beta)\coloneqq  \lim_{n\to \infty} \frac{1}{n} \log W^{\beta}_n= \frac{1}{n} \bbE \left[ \log W^{\beta}_n\right].
\end{equation}
We say that \textit{very strong disorder} holds when $\f(\beta)<0$ (that is to say when $ W^{\beta}_n$ converges exponentially fast to zero). It was established in \cite{CY06} that the ``strength of disorder'' is monotone in $\beta$ in the sense that there exist $\beta_c$ and $\bar \beta_c$ in $[0,\infty]$ such that:
\begin{itemize}
 \item [(a)] Weak disorder holds when $\beta<\beta_c$ and strong disorder holds when $\beta>\beta_c$,
\item [(b)] Very strong disorder holds if and only if  $\beta>\bar \beta_c$.
 \end{itemize}
 The weak and strong disorder regimes correspond to different asymptotic behavior  of $X$ under $P^\beta_{\go,n}$ as $n\to \infty$.
When weak disorder holds, in the simple random walk case (cf.\ \cite{CY06}, see also \cite{H24} for a recent short proof) the scaling limit of $(X_k)^n_{k=1}$ under $P^{\beta}_{\go,\beta}$ on the diffusive scale is the same as  under $P$, i.e., a Brownian Motion.
Analogous results have been proved in the case when $X$ is in the domain of attraction of an $\alpha$-stable law \cite{W16}.

\smallskip

On the contrary, when strong disorder holds, the polymer is believed to exhibit a different behavior: the trajectories should \textit{localize} around a favorite corridor along which the environment $\go$ is particularly favorable. This conjectured localized behavior has been corroborated by several mathematical results \cite{CH02,CSY03,CH06,BC20}, the stronger results being obtained under the assumption of \textit{very strong disorder} and additional technical restrictions.

\smallskip

This distinction between strong and very strong disorder is however not crucial: it was conjectured in \cite{CH06,CY06} that the two critical points $\beta_c$ and $\bar \beta_c$ coincide, and this conjecture was proved to hold true in \cite{JL24} under the assumption that the disorder is upper bounded. The present paper removes this assumption and extends the result to arbitrary random walks.
We refer to the introduction of \cite{JL24} and to the recent survey \cite{Zy24} for a more detailed discussion on the localization transition.

\smallskip

 Let us also mention -- although the present paper does not bring any new perspective on the topic -- that, in the case where $X$ is the nearest neighbor random walk in $\bbZ^d$, beyond the change from delocalization to localization, the critical point $\bar \beta_c$ is also expected to mark a transition from diffusive to superdiffusive behavior. This is captured by the transversal fluctuation exponent $\xi$, defined informally through the relation $E^{\go,\beta}_n[|X_n|^2]=n^{2\xi+o(1)}$ under the polymer measure $P_{\omega,n}^\beta$. In the  weak disorder phase, the invariance principle implies $\xi=\frac{1}{2}$ while in the strong disorder phase it is expected that the polymer becomes super-diffusive, with $\xi>\frac{1}{2}$. Proving this for the standard model is one of the major open problems in this field.

\smallskip To our knowledge, superdiffusivity results have only been obtained for models for which it is known that $\beta_c=0$. For the DPRE introduced above, it is predicted that $\xi=\frac{2}{3}$ for all $\beta>0$ when $d=1$, a conjecture which has proved to hold true in the specific case  $\log$-$\gG$ distributed environment in \cite{S12}. Upper and lower bounds for $\xi$ have been 
achieved in \cite{P00,Mej04} for a polymer model in which both the environment and the random walk are Gaussian. In addition,  results have been obtained for DPRE with heavy-tailed environments \cite{AL10,DZ16,BT19} as well as related models where the environments displays long-range correlations \cite{L11,L12_1,L12_2}.
The question of finding a directed polymer model for which $\xi=1/2$ for small values of $\beta$ and $\xi>1/2$ for large values of $\beta$ remains widely open.

\smallskip

To facilitate the discussion of the results, let us finally introduce the $L^2$ critical point
 \begin{equation}\label{defbeta2}
  \beta_2\coloneqq  \sup\left\{ \beta\ge 0 \colon \left(e^{\gl(2\beta)-2\gl(\beta)}-1\right) \sum_{k\ge 1} P^{\otimes 2}\left( X^{(1)}_k=X^{(2)}_k\right)\le  1 \right\},
 \end{equation}
 where $P^{\otimes 2}$ is the law of two independent copies $X^{(1)}$ and $X^{(2)}$ of the random walk $X$.
 It is not difficult to check that $\sup_{n\geq 0}\bbE\left[ (W_{n}^\beta)^2\right]<\infty$ if and only if $\beta\in [0,\beta_2)\cup \{0\}$, which implies that $(W_n^\beta)_{n\in\bbN }$ converges in $L^2$ for $\beta<\beta_2$ and hence $\beta_c\ge \beta_2$. We also note the equivalence
\begin{equation}\label{lequiv}
 \beta_2=0   \quad \iff  \quad (X^{(1)}_k-X^{(2)}_k)_{k\ge 0} \quad \text{ is recurrent}.
\end{equation}
Historically, the $\beta_2$ critical point was used as a sufficient criterion to ensure $\beta_c>0$ in the case when  $X^{(1)}-X^{(2)}$ is transient (in particular in the simple random walk case in $d\ge 3$, see \cite{IS88,B89}). In this paper, we show that this condition is also necessary in the sense that $\beta_c=0$ whenever $X^{(1)}-X^{(2)}$ is recurrent.
We refer the reader to \cite{CSZ21} for a previous work exploring the relation between $\beta_2=0$ and $\beta_c=0$  considering a more general setup in which  $X$ is only assumed to be a Markov chain on a countable state space.

\section{Results}

\subsection{Coincidence of the critical points}

 As mentioned earlier, it had been conjectured in \cite{CH06,CY06} that there is a sharp transition from weak to very strong disorder, in other words that  $\beta_c=\bar \beta_c$. This conjecture was formulated in the case where the reference walk $P$ is the simple random walk on $\bbZ^d$, and it was recently proved to hold true under the additional assumption that the environment  $\go$ is upper bounded \cite{JL24}.
We extend the validity of this result by relaxing the assumptions on the random walk and on the environment (recall \eqref{taileta}).

\begin{theorem}\label{firstmain}
 If $\eta>0$, then $\beta_c=\bar \beta_c$.
 Furthermore, if $\beta_c>\beta_2$ then weak disorder holds at $\beta_c$.
\end{theorem}

\begin{remark}
The assumption $\eta>0$ is necessary and the result may fail to hold when the assumption is violated. For instance, it was proved in \cite{V23}, that for a very heavy-tailed one-dimensional random walk, one may have $\bar \beta_c=\infty$ and $\beta_c\in (0,\infty)$. To illustrate this point further, we show in Proposition~\ref{zeroinfinite} that it is even possible to have $\beta_c=0$ and $\bar \beta_c=\infty$.
\end{remark}

In the case of the simple random walk on $\bbZ^d$ with $d\ge 3$, we have $\beta_c>\beta_2$ (see \cite[Section~1.4]{BS10} for $d\ge4$ and \cite[Theorem B]{JL24} for the full details concerning the case $d=3$). This strict inequality is however not always valid. A simple counterexample is that of the simple random walk when $d=1$ or $2$ (for which $\beta_c=\beta_2=0$, and weak disorder trivially holds at $\beta_c$).  We later show that the equality $\beta_c=\beta_2$ is in fact valid  whenever $\beta_2=0$ (see Proposition~\ref{strongbeta}). We may also have $\beta_c=\beta_2$ when $\beta_2>0$, for instance if either $\eta=d=1$ or $\eta=d=2$ (see \cite[Corollary~2.22]{JL24_2}). It is an interesting question whether weak or strong disorder holds at $\beta_c$ in that case.
 
 \smallskip

 To complement our result and to highlight the gap between what we can prove and what we believe to be true, we present a sufficient condition for $\beta_c>\beta_2$. This criterion is related to the tail distribution of the first intersection time of two independent walks
 $$T\coloneqq \inf\left\{n> 0 \colon  X^{(1)}_n=X^{(2)}_n \right\}, $$
with the convention $\inf \emptyset=\infty$. We then define the exponent $\alpha$ by
\begin{equation}\label{defalpha}
 \alpha\coloneqq  -\limsup_{n\to \infty}\frac{\log P^{\otimes 2}\left( T\in[n,\infty\right))}{\log n}\in [0,\infty].
\end{equation}
 
\begin{proposition}\label{alphaundemi}
 If $\beta_2>0$ and $\alpha>1/2$, then $\beta_c>\beta_2$.
\end{proposition}

\begin{remark}
The example with $\beta_c=\beta_2$ from  \cite[Corollary~2.22]{JL24_2} has $\alpha=0$. We believe that one should have $\beta_c>\beta_2$ for all $\alpha>0$, but the ideas used in our proof -- which combine the observations made in \cite[Section~1.4]{BS10} with pinning model estimates adapted from \cite{DGLT} -- clearly stop working whenever $\alpha<1/2$.
Note that $\alpha=\frac d2-1$ for the simple random walk and thus the assumption $\alpha>1/2$ is satisfied when $d\ge 4$ but not for $d=3$.
\end{remark}

 \subsection{Integrability of $W^{\beta}_n$ at criticality}\label{integritix}
 
We defined the \textit{integrability threshold} exponent as 
$$\p(\beta):= \sup \left\{ p\ge 1 \colon  \sup_{n\ge 1}\bbE\left[ (W^{\beta}_n)^p\right]<\infty\right\}.$$
It was introduced in \cite{J22} and provides detailed information concerning the tail behavior of the partition function.
When strong disorder holds, we have $\p(\beta)=1$. On the other hand, by \cite[Corollaries~2.8 and 2.20]{JL24_2}, if $\eta>0$ we have $\p(\beta)>1$ in the weak disorder phase and in that case we have $\bbP[W^{\beta}_\infty\ge u ]\asymp u^{-\p(\beta)}$
(where $f(u)\asymp g(u)$ if $f(u)/g(u)$ is bounded away from $0$ and $\infty$ as $u\to \infty$).

\smallskip As observed in \cite{JL24}, the fact that weak disorder holds at $\beta_c$ combined with results from \cite{J23} allows to deduce the value of $\p$ at the critical point. The proof of the following result is identical to the one found in \cite{JL24}. It relies on on an extension of \cite[Corollary 1.3]{J22} to the case of  unbounded $\go$ which is proved in \cite{JL24_2}.
\begin{corollary}\label{valueofpc}
 Assuming that $(X_k)_{k\ge0}$ is the simple random walk on $\bbZ^d$ and $d\ge 3$, we have $\p(\beta_c)=1+\frac{2}{d}.$
\end{corollary}

 \begin{proof}
 From \cite[Theorem 1.2]{J23} it is known that $\beta \mapsto \p(\beta)$ is right-continuous at points where $p^*(\beta)\in(1+\frac{2}{d},2]$. Since $\p$ jumps from $p^*(\beta_c)\  >1$ to $1$ at $\beta_c$ and since $\beta_c>\beta_2$ implies that $\p(\beta_c)\le 2$, we necessarily have $p^*(\beta_c)\le 1+\frac{2}{d}$.
On the other hand, from \cite[Corollary~2.8]{JL24_2} we have $p^*(\beta_c)\ge 1+\frac{2}{d}$.
\end{proof}

Corollary \ref{valueofpc} can be extended beyond the case of the simple random walk. We let  $D$ denote \textit{the transpose of} the transition matrix of $X$ (defined by  $D(x,y)\coloneqq P(X_1=x-y)$).
With a small abuse of notation we set 
 \begin{equation}\label{defnu}\begin{split}
 \|D^k \|_{\infty}&:=\max_{x\in \bbZ^d} D^k(0,x)=\max_{x\in \bbZ^d} D^k(x,0) =\max_{x\in \bbZ^d} P (X_k=x),\\
  \nu&:=-\limsup_{k\to \infty} \frac{\log \|D^k \|_{\infty}}{\log k}.
 \end{split}\end{equation}
Let us note that if  $X^{(1)}-X^{(2)}$ is transient, $\eta$, $\alpha$ and $\nu$ (recall \eqref{taileta} and \eqref{defalpha}) satisfy the following inequality (we provide a proof in Appendix \ref{apalf} for completeness)
\begin{equation}\label{etaalphanu}
    \nu \le  \alpha+1 \le\frac{d}{2\wedge \eta}.
 \end{equation}

\begin{proposition}\label{forallwalks}
If $\beta_2>0$ and weak disorder holds at $\beta_c$, then we have
 \begin{equation}
  \p(\beta_c)\in \left[ 1+ \frac{2 \wedge \eta}{d}, 1+\frac{1}{\nu \vee 1}  \right].
 \end{equation}
\end{proposition}

\begin{remark}
{ From Theorem \ref{firstmain}, the assumption made in Proposition \ref{forallwalks} is satisfied when $\beta_c>\beta_2$, but we do not require that latter condition. At the moment it is not known whether weak disorder holds at $\beta_2$ in general. Moverover,}\ the  interval $\left[ 1+ \frac{2 \wedge \eta}{d}, 1+\frac{1}{\nu \vee 1}  \right] $ is ill-defined if $d=1$ and $\eta>1$ but in that case  the walk $X^{(1)}-X^{(2)}$ is recurrent and  $\beta_2=0$ (recall \eqref{lequiv}).
\end{remark}

 \begin{remark}
There are plenty of examples for which $\nu=  d/(2\wedge \eta)=\alpha\ { + }\ 1$. This includes all $d$-dimensional random walks where $X_1$ has a finite second moment and a support that generates $\bbZ^d$ (or a subgroup of finite index), and symmetric $1$-dimensional walks satisfying $P(X_1=x)= |x|^{-1-\eta+o(1)}$ as $x\to \infty$. In these cases Proposition \ref{forallwalks} allows to identify the value of $p^*(\beta_c)$ if one assumes that weak disorder holds at $\beta_c$.
 \end{remark}
 
\subsection{Absence of phase transition in the recurrent case}

It was established in \cite{CH02, CSY03} that $\beta_c=0$ in dimension $1$ and $2$ (see also \cite{CV06} and \cite{L10} for  proofs that $\bar \beta_c=0$ when $d=1$ and $d=2$ respectively). These results -- and their proofs -- suggest that there is no weak disorder phase as soon as $\beta_2=0$. We provide a proof of this statement under the minimal assumption that $\go$ admits some finite exponential moments,
\begin{equation}\label{somexp}
 \left\{\beta>0 \ \colon\ \bbE\left[ e^{\beta \go_{1,0}}\right]<\infty \right\}\ne \emptyset.
\end{equation}
Note that $W^{\beta}_n$ is only defined for $\beta$ such that $\gl(\beta)<\infty$ in that case.

\begin{theorem}\label{strongbeta}
Assume that \eqref{somexp} holds. If $\beta_2=0$, then $\beta_c=0$.
\end{theorem}

For the remainder of this section, we return to assuming \eqref{allorder}.
Combining the above and Theorem~\ref{firstmain}, we obtain as a corollary that very strong disorder holds for all $\beta$ provided that the power-tail assumption is satisfied.
\begin{corollary}\label{strongibeta}
  If $\eta>0$ and $\beta_2=0$, then $\bar \beta_c=0$. 
\end{corollary}
To illustrate the necessity of the assumption $\eta>0$ for this last result, we present an example of a polymer model for which $\beta_2=0$ and $\bar \beta_c=\infty$.
We define the tower sequence $(a_k)_{k\ge 0}$ by $a_0=1$ and 
$a_{k+1}\coloneqq 2^{a_k}$ and a function $f(x)$ on $\bbZ$ by $f(x)=1$ if $|x|\leq 1$ and
\begin{equation}\begin{split}
       f(x)&= \frac{1}{(2a_k+1) a_{k-1}^4} \quad  \text{ if } |x|\in (a_{k-1},a_k], \quad  k\ge 1.
                \end{split}
\end{equation}
With the above definition, we have $\sum_{x\in \bbZ} f(x)\le 5$. We can thus define $ g(x)= \frac{f(x)}{\sum_{y\in \bbZ} f(y)}$
and consider a simple random walk on $\bbZ$ whose increment distribution satisfies
$P\left( X_1=x\right)= g(x).$

\begin{proposition}\label{zeroinfinite}
 For a directed polymer based on the above random walk, we have $\beta_2=\beta_c=0$ but  $\f(\beta)=0$ for every $\beta\ge 0$.
\end{proposition}

\subsection{Organization of the paper}
In Section~\ref{toolbox}, we present three technical results. These results are adapted from \cite{JL24} and are used to prove Theorem~\ref{firstmain}, Proposition~\ref{alphaundemi} and Theorem~\ref{strongbeta}

\medskip

In Sections~\ref{orgamain} and~\ref{discretos}, we prove Theorem~\ref{firstmain}. The proof largely follows the reasoning used in \cite{JL24} to treat the case of upper-bounded environment but a couple of technical innovations are required to deal with an unbounded environment and general random walks. While we shortly recap some of the main ideas, we direct the interested reader to \cite{JL24} for more in depth explanation of the proof mechanism.

\medskip

In Section~\ref{sbproof}, we prove Theorem~\ref{strongbeta} using the material from Section~\ref{toolbox}.

\medskip

The proof of Proposition~\ref{alphaundemi} is based on an observation made in \cite[Section 1.4]{BS10} and on an adaption of the methods used in \cite{DGLT}, which is developed in Appendix~\ref{proofalpha}.

\medskip

One of the bounds in Proposition \ref{forallwalks} can be derived directly from \cite[Corollary 2.20]{JL24_2}, but the other one requires an extension of \cite[Theorem 1.1]{J23} to the case of an arbitrary random walk. This is done in Appendix \ref{apalf}.

\medskip

Finally the proof of Proposition \ref{zeroinfinite}, which is a result of illustrative value, is detailed in Appendix \ref{proofzero}.

\section{A toolbox of preliminary results}\label{toolbox}

We present here a couple of technical results required for our proof of Theorems~\ref{firstmain} and ~\ref{strongbeta}. These results can be found in \cite[Section 3]{JL24}. We present them here with a couple of key modifications to fit the setup of the present paper.

\subsection{A finite volume criterion relying on fractional moments}

By Jensen's inequality we have, for any $\theta\in (0,1)$,
\begin{equation}\label{tops}
 \bbE[ \log W_n]=  \theta^{-1}\bbE[ \log (W_n)^\theta]\le \theta^{-1} \log \bbE\left[ (W_n)^\theta\right].
\end{equation}
Hence, to show that very strong disorder holds, it is sufficient to show that $\bbE\left[ (W_n)^{\theta}\right]$ decays exponentially fast in $n$.
Adapting the argument used to prove \cite[Theorem~3.3]{CV06}, we show that such an exponential decay holds if there exists some $n$ such that $\bbE\left[W_n^{\theta}\right]$ is smaller than a large power of $n$.
To make this statement precise, we need to fix the value of $\theta$ so, recalling \eqref{taileta}, we set  (throughout the paper we use the notation $a\wedge b=\min(a,b)$  and $a\vee b=\max(a,b)$ for $a,b\in \bbR$)
\begin{equation}\label{thetak}
\bar \eta\coloneqq  \eta\wedge 1,\quad   \theta\coloneqq 1-\frac{\bar \eta}{4d} \quad \text{ and } \quad   K=\frac{20d}{\bar \eta},
\end{equation}

\begin{proposition}\label{finitevol}
Assume that $\eta>0$. There exists $n_0$  such that, if for some  $n\ge n_0$ we have
\begin{equation}\label{consequence}
 \bbE\left[ (W_n)^{\theta}\right]< 2n^{-K},
\end{equation}
then very strong disorder holds.
\end{proposition}
\begin{proof}
From \eqref{tops}, we have
\begin{equation}\label{keyz}
 \f(\beta)\le \liminf_{m\to \infty} \frac{1}{\theta nm} \log \bbE\left[ (W_{nm})^{\theta}\right].
\end{equation}
  Given $x_1,\dots,x_m\in \bbZ^d$, we let $\hat W_{nm}(x_1,x_2,\dots,x_m)$ denote the contribution to the partition function of
 trajectories that go through $x_1,x_2,\dots, x_m$ at times $n,2n,\dots, nm$, that is
$$\hat W_{nm}(x_1,x_2,\dots,x_m)\coloneqq  E\left[ e^{\beta \sum_{k=1}^{nm} \go_{k,X_k}-nm\gl(\beta)}\ind_{\{\forall i\in \lint 1,m\rint,\ X_{ni}=x_i\}}\right].$$
 The case $m=1$ defines the point-to-point partition function,
\begin{align}\label{p2p}
\hat W_n^\beta(x) = E\left[ e^{\beta \sum_{k=1}^{n} \go_{k,X_k}-n\gl(\beta)}\ind_{\{X_{n}=x\}}\right],
\end{align}
which will play an important role later. Now,
using the inequality
\begin{equation}\label{subadd}
	\left(\sum_{i\in I} a_i\right)^{\theta}\le \sum_{i\in I} a_i^{\theta},
\end{equation}
 which is valid for an arbitrary collection of non-negative numbers $(a_i)_{i\in I}$ and $\theta\in (0,1)$ (in the remainder of the paper we simply say \textit{by subadditivity} when using \eqref{subadd}), we obtain
\begin{equation}\begin{split}\label{factor}
 \bbE\left[ (W_{nm})^{\theta}\right]&\le \sum_{(x_1,\dots,x_m)\in (\bbZ^d)^m} \bbE\left[\left(\hat W_{nm}(x_1,x_2,\dots,x_m)\right)^{\theta} \right]\\
 &= \sum_{(x_1,\dots,x_m)\in (\bbZ^d)^m} \prod_{i=1}^m
 \bbE\left[\left(\hat W_{n}(x_{i}-x_{i-1})\right)^{\theta} \right]\\
 &=\left(\sum_{x\in \bbZ^d}\bbE\left[ \left( \hat W_n(x)\right)^{\theta}\right]\right)^m
\end{split}\end{equation}
where the factorization in the  second line is obtained by combining the Markov property for the random walk with the independence of the environment.
In order to conclude using \eqref{keyz}, we need to show that under the assumption \eqref{consequence} we have $\sum_{x\in \bbZ^d}\bbE\left[ \left( \hat W_n(x)\right)^{\theta}\right]<1$.

\medskip

We set $R=R_n\coloneqq n^{\frac{ 16}{\bar \eta}}$.
Using the inequality $\hat W_n(x)\le W_n$ for $|x|\le R$ and Jensen's inequality for $|x|>R$, we obtain
\begin{equation}\label{sumoftwo}
 \sum_{x\in \bbZ^d}\bbE\left[ \left( \hat W_n(x)\right)^{\theta}\right]\le (2R+1)^d \bbE\left[ (W_n)^\theta\right]
 + \sum_{|x|> R}P(X_n=x)^{\theta}.
\end{equation}
Let us show that both terms on the r.h.s.\  of \eqref{sumoftwo} are smaller that $\frac 13$ for $n$ sufficiently large.
For the first term, it is simply a consequence of the assumption \eqref{consequence} and our choice for $R$.
To bound the second term, for any $k\ge 0$, we use Jensen's inequality for the uniform measure on the annulus $\{x\in\bbZ^d\colon |x|\in (2^{k} R,2^{k+1} R]\}$ to obtain
\begin{equation}\label{jensos}
 \sum_{|x|\in (2^{k} R,2^{k+1} R]}P(X_n=x)^{\theta}\le (2^{k+2}R)^{d(1-\theta)}   P\left(|X_n|\in (2^{k} R,2^{k+1} R]\right)^\theta\ .
\end{equation}
Recalling \eqref{taileta}, we have $\bbP(|X_1|\ge u)\le u^{-\bar \eta/2}$ for $u$ sufficiently large.
Hence we obtain that  
\begin{equation}\label{aveck}
 P\left(|X_n|\in (2^{k} R,2^{k+1} R]\right) \le P(|X_n|> 2^k R)\le n P\left( |X_1|> (2^{k} R/n)\right) \le (2^{k} R)^{-\bar \eta/2} n^{1+\bar \eta/2}.
\end{equation}
Since by choice $d(1-\theta)- \theta\bar \eta/2=-\bar\eta/4+ \bar\eta^2/8d\leq -\bar\eta/8<0$, by combining \eqref{jensos} and \eqref{aveck} and summing over $k$ we obtain that there exists $C>0$ such that
\begin{equation}
	\sum_{|x|> R}P(X_n=x)^{\theta}\le C n^{ \theta(1+\bar \eta/2)} R^{-\frac{\bar \eta}{8}}= C n^{ \theta(\frac{\bar\eta}2 +1)-2}.
  \end{equation}
Since $\bar\eta, \theta\leq 1$, this concludes the proof.
\end{proof}

\subsection{Bounding the fractional moment using the size-biased measure}\label{sizeb}

We introduce the size-biased measure for the environment  defined by
$ \tilde \bbP_n(\dd \go)\coloneqq  W^{\beta}_n\bbP(\dd \go).$
Intuitively, strong disorder holds when $\tilde \bbP_n$ and $\bbP$ are ``asymptotically singular'' (meaning that there exist typical events under $\bbP$ that become increasingly untypical under $\tilde \bbP_n$ as $n$ grows). The following result (which is a reformulation of \cite[Lemma 3.2]{JL24}) is a quantitative version of this statement. 
\begin{lemma}\label{babac}
 For any measurable event $A$ and  $\theta\in (0,1)$,
 \begin{equation}
  \bbE\left[ (W^{\beta}_n)^{\theta} \right] \le \bbP(A)^{(1-\theta)} + \tilde \bbP_n(A^{\complement})^{\theta},
\end{equation}
\end{lemma}
\begin{proof}
 We adapt the proof  found in  \cite[Lemma 2.2]{BCT25}.
We split the expectation in two and then bound the first part using H\"older's inequality and the second part using Jensen's inequality,
 \begin{equation}
	 \bbE\left[ (W^{\beta}_n)^{\theta} \right]=\bbE\left[ (W^{\beta}_n)^{\theta}\ind_A \right]+ \bbE\left[ (W^{\beta}_n)^{\theta}\ind_{A^{\complement}} \right] \le   \bbE\left[W^{\beta}_n\right]^{\theta} \bbP(A)^{(1-\theta)}
  + \bbE\left[ W^{\beta}_n\ind_{A^{\complement} }\right]^{\theta},
 \end{equation}
which gives the desired result.
\end{proof}

To prove that $\f(\beta)<0$, our strategy is to combine Proposition~\ref{finitevol} with Lemma~\ref{babac} and find an event $A_n$ which is unlikely under the original measure $\bbP$ and typical under the size-biased measure $\tilde \bbP_n$.

\subsection{Spine representation for the size-biased measure}\label{spineR}

In this section, we recall a well-known representation for the size-biased measure.
We define $(\hat \go_i)_{i\ge 1}$ as a sequence of i.i.d. random variable -- whose distribution is denoted by $\hat \bbP$ -- with marginal distribution given by
\begin{equation}\label{tilted}
\hat \bbP[  \hat \go_1\in \cdot]= \bbE\left[ e^{\beta\go_{1,0}-\gl(\beta)}\ind_{\{\go_{1,0}\in \cdot\}}\right].
\end{equation}
 and $X$ a random walk with distribution $P$. Note that $\hat\bbP$ and $\hat\omega$ are unrelated to the notation $\hat W_n^\beta(x)$ for the point-to-point partition function introduced in \eqref{p2p}. Given $\go$, $\hat \go$ and $X$, all sampled independently, we define a new environment $\tilde \go=\tilde \go(X,\go,\hat \go)$  by
\begin{equation}
 \tilde \go_{i,x}\coloneqq \begin{cases}
	 \go_{i,x}  & \text{ if } x\ne X_i,\\
	 \hat \go_i  & \text{ if } x= X_i.
            \end{cases}
\end{equation}
In words,  $\tilde \go$ is obtained by tilting the distribution of the environment on the graph of $(i,X_i)_{i=1}^\infty$.
The following results states that the distribution of $\tilde \go$ corresponds to the size-biased measure. We refer to \cite[Lemma~3.3]{JL24} for comments on and a proof of the following classical statement.
\begin{lemma} \label{sbrepresent}It holds that
\begin{equation}
\tilde \bbP_n\left(\ (\go_{i,x})_{i\in \lint 1, n\rint, x\in \bbZ^d }\in \cdot \ \right)=P\otimes \bbP\otimes \hat \bbP\left(  (\tilde \go_{i,x})_{i\in \lint 1, n\rint, x\in \bbZ^d } \in \cdot \right).
\end{equation}
\end{lemma}
In the course of our proof we refer to the above as the \textit{spine representation} of the size-biased measure. Under $\tilde\bbP_n$, the distribution of the environment has been tilted along a random trajectory $X$, which we refer to as the \textit{spine}.

\section{Organization of the proof of Theorem~\ref{firstmain}}\label{orgamain}

Let us start by reformulating the result. We want to show that the following implication holds
\begin{equation}\label{keyimplic}
 \bbP(W^{\beta}_{\infty}=0)=1 \quad \text{ and }\quad \beta>\beta_2 \quad  \Longrightarrow  \quad \f(\beta)<0.
\end{equation}
It is a simple task to check that \eqref{keyimplic} implies both statements in Theorem~\ref{firstmain}.

\subsection{Identifying the right event}
Recalling \eqref{thetak}, we introduce a new family of parameters. We set
\begin{equation}\label{choiceK3}
K_2\coloneqq  \frac{2}{\eta}\left(\frac{K}{1-\theta}+1 \right)+1 \quad \text{ and } \quad  K_3\coloneqq  1+dK_2+ \frac{2K}{1-\theta}.
\end{equation}
Using a union bound, we have, for $n$ sufficiently large,

\begin{equation}\label{standardo}
	P\left( \max_{k\in \lint 1, n \rint} |X_{k}|\ge n^{K_2} \right)\le n P\left(  |X_1|\ge n^{K_2-1} \right) \le\frac 12 n^{-\frac{K}{1-\theta}}.
\end{equation}
We introduce the shifted environment $\theta_{n,z}\go$ by setting 
\begin{equation}\label{definishift}
(\theta_{n,z} \go)_{k,x}\coloneqq  \go_{n+k,z+x}
\end{equation}
and let $\theta_{n,z}$ act on functions of $\go$ by setting $\theta_{n,k}f(\go)= f(\theta_{n,z} \go)$.
\begin{proposition}\label{spot}
Assume that strong disorder holds and $\beta>\beta_2$. Then there
exist $C>0$ and $n_0\in\N$ such that for all $n\ge n_0$ there exists $s=s_n\in  \lint 0,C \log n\rint$ such that, setting
$$ A_n\coloneqq  \left\{\exists (y,m)\in \lint 0, n-s\rint\times \lint -n^{K_2},n^{K_2}\rint^d\colon \theta_{m,y} W^{\beta}_{s} \ge n^{K_3} \right\}, $$
we have 
\begin{equation}\begin{split}
\bbP(A_n)\le n^{-\frac{K}{(1-\theta)}} \quad \text{ and } \quad  \tilde \bbP_n(A^{\complement}_n)\le n^{-\frac{K}{\theta}}.
\end{split}\end{equation}

\end{proposition}

 \begin{proof}[Proof of Theorem~\ref{firstmain} assuming Proposition~\ref{spot}]
 We assume that strong disorder holds and that $\beta>\beta_2$.
 Combining Lemma~\ref{babac} and Proposition~\ref{spot}, we obtain that for $\theta$ and $K$ given by \eqref{thetak}  we have, for any $n\ge n_0$, 
 \begin{equation}
  \bbE\left[ (W^{\beta}_n)^{\theta}\right]\le  \left(n^{-\frac{K}{(1-\theta)}}\right)^{1-\theta}+ \left(n^{-\frac{K}{\theta}}\right)^{\theta} = 2n^{-K},
 \end{equation}
 and we can conclude the proof of \eqref{keyimplic} using Proposition~\ref{finitevol} for $n$ sufficiently large.\qedhere
 \end{proof}

Bounding $\bbP(A_n)$ is not difficult (and the bound is valid for any choice of $s$).
Indeed, by translation invariance for $\go$, a union bound and Markov's inequality, we have
\begin{equation}\label{pan}
 \bbP(A_n)\le  n (2n^{K_2}+1)^d\bbP\left[ W^{\beta}_s\ge n^{K_3}\right]\le   n (2n^{K_2}+1)^d n^{-K_3}\le n^{-\frac{K}{\theta}}
\end{equation}
where the last inequality is valid for $n$ sufficiently large, due to  our choice parameters \eqref{choiceK3}.

\medskip

The harder problem is to estimate $\tilde \bbP_n(A^{\complement}_n)$. The strategy we use for this is the same as in \cite{JL24}.
The main task is to obtain a good lower bound on  $\bbP(W^{\beta}_{s} \ge n^{K_3})$ (and hence on  $\tilde \bbP_s(W^{\beta}_{s} \ge n^{K_3})$)  in the strong disorder regime.
 For this, we extend the validity of \cite[Proposition~4.2]{JL24} (proved in \cite{JL24} only for upper bounded disorder and the simple random walk).

\medskip

\begin{proposition}\label{almostthesame}
If strong disorder holds and $\beta>\beta_2$, then for any $\gep>0$ there exist $C=C(\beta,\gep)>0$ and $u_0=u_0(\beta,\gep)>1$ such that every $u\ge u_0$,
 \begin{equation}\label{needed2}
 \exists s\in \lint 0,C \log u\rint, \quad  \bbP\left[  W^{\beta}_s\ge u\right]\ge u^{-(1+\gep)}.
\end{equation}

\end{proposition}

Note that $\bbP\left[  W^{\beta}_s\ge u\right]\ge u^{-(1+\gep)}$ immediately implies that
\begin{equation}\label{forthetilt}
 \tilde\bbP_s\left[  W^{\beta}_s\ge u\right]= \bbE\left[ W^{\beta}_s \ind_{\{ W^{\beta}_s\ge u\}}\right]\ge u^{-\gep}.
\end{equation}

The proof follows the same road map as in \cite{JL24} for but several technical adaptation are required.
This is detailed in the next subsections.

 \subsection{Tail distribution of the maximum of the point-to-point partition function}

The starting point in proving Proposition~\ref{almostthesame} is to establish that the tail distribution of $\max_{n\ge 0}  W^{\beta}_n$  decreases like $u^{-1}$ in the strong disorder regime. In the case of upper-bounded disorder, this is an easy consequence of the martingale stopping theorem (see \cite[Lemma~4.3]{JL24}).
Without this assumption, the proof is a more subtle but it has already been obtained in \cite[Corollary~2.6 and Proposition~2.19]{JL24_2} (we consider the special case when $p^*(\beta)=1$).

\begin{lemma}\label{dull}
If strong disorder holds, then there exists $c>0$ such that for any $u\ge 1$,
\begin{equation}
 \bbP\left( \exists n\geq 0\colon W^{\beta}_n\geq  u\right) \in \left[\frac{c}{u},\frac{1}{u}\right].
 \end{equation}
\end{lemma}
\begin{proof}
The lower bound is due to \cite[Corollary~2.6]{JL24_2} and the upper bound holds for any non-negative martingale due to Doob's martingale inequality.
\end{proof}

To prove Proposition~\ref{almostthesame}, we show that a similar lower bound holds for
the maximum of \textit{point-to-point} partition functions  introduced in \eqref{p2p}.

\begin{theorem}\label{locaend}
If strong disorder holds and $\beta>\beta_2$, then there exists $c>0$ such that for all $u\ge 1$
\begin{equation}\label{wopwop}
 \bbP\left[ \sup_{n\ge 0, x\in \bbZ^d} \hat  W^{\beta}_n(x)\ge u\right] \ge \frac{c}{u}.
 \end{equation}
\end{theorem}

\begin{proof}[Proof of Proposition~\ref{almostthesame}]
The inequality in \eqref{needed2} is deduced from \eqref{wopwop} in
\cite[Section 6]{JL24}, without relying on  the assumption that the environment is upper-bounded. That argument also does not rely on the specifics of the simple random walk,  so we merely sketch the proof. We set 
\begin{align*}
A_{v,M}:=\left\{ \ \sup_{n\in\lint 0,M\rint, x\in \lint -M,M\rint^d} \hat  W^{\beta}_n(x)\ge v \right \}
\end{align*}
to be a truncated analogue of the event considered in \eqref{wopwop}. For fixed $v>1$, using \eqref{wopwop} we can find $M(v)$ such that $\mathbb P(A_{v,M(v)})\geq \frac{c}{2v}$. On $A\coloneqq A_{v,M(v)}$, we let $(T,Y)$ the  minimal element (for the lexicographical order) in $\llbracket 1,M\rrbracket\times\llbracket-M,M\rrbracket^d$ such that $\hat W_T^\beta(Y)\geq v$.
Choosing $(T,Y)$ minimal guarantees that the shifted environment $\theta_{T,Y} \go$ is independent of the past and distributed like $\go$.
Note that on the event $A\cap \theta_{T,Y}A$ we  have $\max_{(t,y)\in\llbracket 1,2M\rrbracket\times\llbracket-2M,2M\rrbracket^d} \hat W_t(y)\geq v^2$ and by independence $\bbP(  A\cap \theta_{T,Y}A)=\bbP(A)^2\ge (c/2v)^2$. Repeating this argument, we find that
$$ \bbP\left(\exists n\leq kM,\ W^\beta_{n}\geq v^k \right)\geq \left(\frac c{2v}\right)^{k}\ge  v^{-k(1+\eps/2)},$$ where the last inequality holds for $v$ large enough (depending on $\eps$). From this lower bound, \eqref{needed2} is deduced easily (with $u_0=v$, $C=\frac{2M}{\log v}$) by replacing $u$ by $v^k$ with $k={\lceil \log_v u\rceil}$ and observing that
\begin{align*} \max_{s\in \lint 1,C \log u\rint} \bbP\left[   W^{\beta}_s\ge u\right]  \ge \frac{1}{C\log u}   \bbP\left[ \exists s\in \lint 1,C \log u\rint\colon   W^{\beta}_s\ge u\right].\tag*{\qedhere}
\end{align*}
\end{proof}

To conclude, we need to show that Proposition~\ref{spot} follows from Proposition~\ref{almostthesame} and  we need\ to prove Theorem~\ref{locaend}. The first task is accomplished in the next  subsection. The proof of Theorem~\ref{locaend} is technically more demanding and is detailed in Section~\ref{discretos}.

\subsection{Proof of Proposition~\ref{spot}}\label{proofspot}

The following adapts the argument presented \cite[Section~6]{JL24} to the setup of a generic random walk.

\medskip

Recall from equation \eqref{pan} that the required bound on $\bbP(A_n)$ is valid for any choice of $s$. We can thus focus on bounding $\tilde \bbP(A^{\complement}_n)$ assuming that  Proposition~\ref{almostthesame} holds.
We apply Proposition~\ref{almostthesame} for $u=n^{K_3}$  and $\gep=1/(3K_3)$ and we consider $s\in \lint 0, C' \log n\rint$, which is such that (recall \eqref{forthetilt})
\begin{equation}\label{sbes}
 \tilde \bbP_s [W^\beta_s\ge n^{K_3}]\ge n^{-K_3\gep}=n^{-1/3}.
\end{equation}
Using the spine representation, we wish to bound  $P\otimes \hat \bbP \otimes \bbP\left[ \tilde \go \in A^{\complement}_n \right]$.
Clearly, we have $\{\tilde \go \in  A^{\complement}_n \} \subset B^{(1)}_n\cup B^{(2)}_n$, where
\begin{equation}\begin{split}
 B^{(1)}_n&\coloneqq  \left\{  \max_{k\in \intp{1,n}} |X_k|> n^{K_2} \right\},\\
B^{(2)}_n&\coloneqq  \left\{  \forall i\in \intp{ 0, \lfloor n/s\rfloor -1}\colon \theta_{is, X_{is}} \tilde W^{\beta}_s<   n ^{K_3} \right\},
\end{split}\end{equation}
and where we recall that $(X_k)_{k\geq 0}$ denotes the spine.
Now from \eqref{standardo} we have 
$$ P\otimes \hat \bbP \otimes \bbP(B^{(1)}_n)=P(B^{(1)}_n)\le  \frac 12n^{-\frac{K}{1-\theta}}.$$To estimate the probability of $B^{(2)}_n$, we observe that, by construction, the variables  $\theta_{is, X_{is}} \tilde W^{\beta}_s$ are i.i.d.\ under $P\otimes \hat \bbP \otimes \bbP$ (see \cite[Lemma 5.1]{JL24} for a proof of this claim). Hence we have
\begin{multline}
	P\otimes \hat \bbP \otimes \bbP\left[ B^{(2)}_n \right]= P\otimes \hat \bbP \otimes \bbP\left[ \tilde W^{\beta}_s<  n^{K_3}  \right]^{\lfloor n/s\rfloor}\\=   \tilde \bbP_s\left[ W^{\beta}_{s} < n^{K_3}  \right]^{\lfloor n/s\rfloor}\le (1-n^{-1/3})^{\lfloor n/s\rfloor}\le e^{-\sqrt{n}},
\end{multline}
where the second equality is a consequence of Lemma~\ref{sbrepresent}, the first inequality comes from \eqref{sbes} and the last (valid for $n$ sufficiently large) from the fact that $s=s_n$ is $O(\log n)$.
Overall, we obtain that, for $n$ sufficiently large,
\begin{equation}
 \tilde \bbP_n(A^{\complement}_n)\le P\otimes \hat \bbP \otimes \bbP\left[ B^{(1)}_n \right]+ P\otimes \hat \bbP \otimes \bbP\left[ B^{(2)}_n \right]\le  \frac 12n^{-\frac{K}{1-\theta}}+ e^{-\sqrt{n}}\le   n^{-\frac{K}{1-\theta}}.\tag*{\qed}
\end{equation}

\section{Proof of Theorem~\ref{locaend}}\label{discretos}

\subsection{Overview}

The reasoning we use to prove  Theorem~\ref{locaend} is analogous to the one used in  \cite{JL24}. We expose here how it can be decomposed in three separate steps. For this we require a couple of notations. We let $\mu_n$ denote the endpoint measure for the polymer and $I_n$ the probability that two independent polymers share the same endpoint. More precisely, recalling $D$ is the transpose of the transition matrix of $X$ (see above \eqref{defnu}) we set 
\begin{equation}\begin{split}\label{defmuin}
 \mu^{\beta}_n(x)&\coloneqq  P^{\beta}_{\go,n}(X_n=x)= \frac{\hat W^{\beta}_{n}(x)}{W^{\beta}_n},\\
 I_n&\coloneqq \sum_{x\in \bbZ^d} (D\mu^{\beta}_{n-1}(x))^2.
\end{split}\end{equation}
We have $D\mu^{\beta}_{n-1}(x)= P^{\beta}_{\go,n-1}(X_n=x)$ (the only  reason why we introduce $D$ as the transpose of the transition matrix is for the convenience of multiplying on the left by $D$ rather than on the right). To justify the expression of $I_n$, let us mention that the quantity naturally appears when computing the bracket increment of $W_n$,
\begin{equation}\label{definechi}
	\bbE\left[(W^{\beta}_n-W^{\beta}_{n-1})^2 \ \middle | \ \cF_{n-1}\right]= \chi(\beta)(W^{\beta}_{n-1})^2 I_n, \quad \text{ where } \chi(\beta)=e^{\gl(2\beta)-2\gl(\beta)}-1.
\end{equation}
We introduce the notation  $I_{(a,b]}\coloneqq  \sum_{n=a+1}^b I_n$ and set
$\tau_u\coloneqq  \inf\{ n \colon  W_n\ge u\}.$ The first step is now to show that if $W_n$ increases from level $u$ to $uK$, the accumulated sum of $I_n$ increases by an amount proportional to $\log K$.
\begin{proposition}\label{keyprop1dis}
If strong disorder holds then
 there exist $\zeta>0$  and $K_0>0$ such that for all $u\ge 1$ and $K>K_0$ we have
 \begin{equation}\label{disversion}
  \bbP\left( \tau_{Ku}<\infty  \ ; \ I_{(\tau_u, \tau_{Ku}]}\le \zeta \log K        \right)\le u^{-1} K^{-2}.
 \end{equation}

\end{proposition}

The second step establishes that with large probability, if $I_{(\tau_u, \tau_{Ku}]}$ is large then $I_n$ cannot be small on the whole interval $(\tau_u,\tau_{Ku}]$.

\begin{proposition}\label{keyprop2dis}
If strong disorder holds and $\beta>\beta_2$, then for any $\zeta>0$ there exists $\delta>0$ such that, for all $u>1$ and  $K\ge K_0$ large enough, 
 \begin{equation}
  \bbP \left( \tau_{Ku}<\infty\ ;\  I_{(\tau_u, \tau_{Ku}]}\ge \zeta \log K \ ; \  \max_{n\in(\tau_u,\tau_{Ku}]}
  I_n\le \delta\right)\le u^{-1} K^{-2}.
 \end{equation}

\end{proposition}

Finally, in a third very short step  we guarantee that with large probability, there are no significant dips 
of $W^{\beta}_n$ between $\tau_u$ and $\tau_{Ku}$.
We set $
 \sigma_{K,u}\coloneqq  \min\{ n \ge \tau_u \colon  W^{\beta}_n\le u/K\}$.

\begin{lemma}\label{keylemmadis}
We have
\begin{equation}
 \bbP( \sigma_{K,u}<\tau_{Ku}<\infty)\le u^{-1} K^{-2}.
\end{equation}

\end{lemma}
\begin{proof}
Note that $\{\sigma_{K,u}<\tau_{Ku}<\infty\}\subset \{\tau_u<\infty \ ; \  \exists n\ge 0, W^{\beta}_{n+\sigma_{K,u}}\ge K^2 W^{\beta}_{\sigma_{K,u}}\}.$
Applying the optional stopping theorem to the martingale $(W^{\beta}_{n+\sigma_{K,u}})$ (considering the filtration $\cF_{n+\sigma_{K,u}}$), we  obtain that on the event $\{\tau_u<\infty\}$ we have
$$\bbP\left[\exists n\ge 0\colon W^{\beta}_{n+\sigma_{K,u}}\ge K^2 W^{\beta}_{\sigma_{K,u}} \middle|   \cF_{\sigma_{K,u}}\right]\le K^{-2}. $$
We conclude by  taking expectation on the event  $\{\tau_u<\infty\}$ and using Lemma~\ref{dull}.
\end{proof}
We can now deduce Theorem~\ref{locaend} from the above. Before presenting the proof, let us explain how Propositions~\ref{keyprop1dis} and~\ref{keyprop2dis} are proved.
Proposition~\ref{keyprop1dis} is the analog of \cite[Proposition~8.1]{JL24} but it requires a different proof since several of the arguments used in \cite{JL24} are specific to upper-bounded disorder.  All the details are provided in Section~\ref{proof1dis}.
On the other hand, the proof of Proposition~\ref{keyprop2dis} is identical to that of \cite[Proposition 8.2]{JL24} and we only provide a sketch of the argument in Section~\ref{dsketch}.

\medskip

\begin{proof}[Proof of Theorem~\ref{locaend}]
We have
\begin{equation} \label{zinclu}
\left\{  \tau_{Ku}<\sigma_{K,u} \   ; \ \max_{n\in(\tau_u,\tau_{Ku}]  }
  I_n> \delta  \right\}\subseteq\left\{ \maxtwo{n\ge 0}{x\in \bbZ^d} \hat W^{\beta}_n(x)\ge \frac{\delta u}{K} \right\}.
  \end{equation}
  Indeed, if $n\in(\tau_u,\tau_{Ku}]$ is such that $I_n\ge \delta$, and $\tau_{Ku}<\sigma_{K,u}$ we have $W^{\beta}_{n-1}\ge u/K$ and thus
\begin{equation}
 \max_{x\in \bbZ^d} \hat W^{\beta}_{n-1}(x)\ge \max_{x\in \bbZ^d} D \hat W^{\beta}_{n-1}(x)
 = W^{\beta}_{n-1} \max_{x\in \bbZ^d} D\mu_{n-1}(x) \ge
 W^{\beta}_{n-1}  I_n \ge \frac{\delta u}{K}.
\end{equation}
We estimate the probability of the l.h.s. event in \eqref{zinclu} as follows: first we observe that
\begin{multline}\label{tryo}
  \bbP\left(  \tau_{Ku}<\sigma_{u,K} \ ; \ \max_{n\in(\tau_u,\tau_{Ku}]  }
  I_n> \delta  \right)\\ \ge   \bbP(  \tau_{Ku}<\infty) -
  \bbP\left( \sigma_{K,u}<\tau_{Ku}<\infty \right)-  \bbP\left( \tau_{Ku}<\infty \ ; \      \max_{n\in(\tau_u,\tau_{Ku}]} I_n\le  \delta\right).
\end{multline}
and then that
\begin{multline}\label{dfg}
 \bbP\left( \tau_{Ku}<\infty \ ; \    \max_{n\in(\tau_u,\tau_{Ku}]} I_n \leq \delta\right) \le  \bbP\left( \tau_{Ku}<\infty \ ; \     I_{(\tau_u,\tau_{Ku}]}\le  \zeta\log K\right) \\+
 \bbP\left( \tau_{Ku}<\infty \ ; \  I_{(\tau_u,\tau_{Ku}]}>  \zeta\log K ;  \max_{n\in(\tau_u,\tau_{Ku}]} I_n\leq \delta\right).
\end{multline}
Using Proposition~\ref{keyprop1dis} and~\ref{keyprop2dis} in \eqref{dfg} we obtain that
\begin{equation}
  \bbP\left( \tau_{Ku}<\infty \ ; \    \max_{n\in(\tau_u,\tau_{Ku}]} I_n\leq\delta \right)\le 2 u^{-1} K^{-2}.
\end{equation}
Then, combining this with Lemma~\ref{dull} and Lemma~\ref{keylemmadis} in \eqref{tryo}, we obtain that
\begin{equation}
  \bbP\left(  \tau_{Ku}<\sigma_{u,K} \ ; \ \max_{n\in(\tau_u,\tau_{Ku}]  }
  I_n> \delta  \right)\ge \frac{c}{Ku}-\frac{3}{ K^{2}u},
\end{equation}
which allows us to conclude using \eqref{zinclu} by taking $K\ge 6/c$.
\end{proof}

\subsection{Proof of Proposition~\ref{keyprop1dis}}\label{proof1dis}
Note that in an unbounded environment, we have to deal with the possibility that $\tau_{uK}=\tau_u$ and thus $I_{(\tau_u,\tau_{uK}]}=0$ (for upper-bounded $\go$ we a priori have $W^{\beta}_{\tau_{u}}\le Lu$ for some fixed $L$ and one may choose $K\ge L$). There is another more serious reason why the argument used  \cite{JL24} requires modification, but to expose it we need to recall it.
The argument relies on the  martingale $\tilde M_n$ defined by $\tilde M_{0}=0$ and
\begin{equation}
 \tilde M_{n+1}=\tilde M_n +\frac{W^{\beta}_{n+1}-W^{\beta}_n}{W^{\beta}_n}.
\end{equation}
We have, by construction, $\tilde M_n-\tilde M_m \ge \log W^{\beta}_n-  \log W^{\beta}_m$ for any $n\ge m$ and the quadratic variation of this martingale is directly related to $I_n$ via the relation
(recall \eqref{definechi})
\begin{equation}
\langle \tilde M \rangle_b-\langle \tilde M \rangle_a= \chi(\beta) I_{(a,b]}.
\end{equation}
We observe that if $W^{\beta}_n$ increases by a multiplicative amount $K$ then $\tilde M$ has to increase by at least $\log K$. Such an increase can occur only if the increase of the bracket $\langle\tilde M\rangle$ is large. In order to derive satisfactory quantitative estimates from this line of reasoning, we rely on Azuma-like concentration results for martingales that require boundedness of the increments of $\tilde M$ (see \cite[Lemma 8.4]{JL24}). In the present setup, the assumption \eqref{allorder} only guarantees that the moments of $(\tilde M_{n+1}-\tilde M_n)_{n\ge 1}$ are all finite and thus such an exponential concentration result cannot hold with full generality.

\medskip

For this reason, we consider, instead of $(\tilde M_n)$, the martingale $(M_n)$ defined as the martingale part in the Doob decomposition of $(\log W_n)_{n\in\bbN}$, that is to say
\begin{equation}
 M_n= \sum^n_{k=1} \left(\log (W^{\beta}_k/W^{\beta}_{k-1})- \bbE\left[ \log (W^{\beta}_k/W^{\beta}_{k-1}) \ | \ \cF_{k-1}\right]\right).
\end{equation}
This martingale enjoys similar properties as $\tilde M$.
For instance, since the conditional expectations in the previous display are non-positive by Jensen's inequality, we have for $n\ge m$
\begin{equation}\label{compared}
 M_{n}-M_m\ge \log W^{\beta}_n-\log W^{\beta}_m
\end{equation}
and the increments of the bracket of $M_n$ are also proportional to $I_n$ (recall the meaning of $\asymp$ introduced at the beginning of Section \ref{integritix}),
\begin{equation}\label{compa}
\langle M\rangle_n-\langle M\rangle_{n-1}\coloneqq    \bbE\left[(M_n-M_{n-1})^2 \ | \ \cF_{n-1} \right]\asymp I_n.
\end{equation}
One of the two inequalities implied by  the $\asymp$ symbol  is proved in \cite{CSY03} and can be deduced directly from \eqref{yas2} in Lemma~\ref{crom} below. The other one is left as an exercise to the interested reader and is not used in the proof.

\medskip

Our strategy is thus to adapt the argument used in \cite{JL24} by applying it to the martingale $(M_n)_{n\in\bbN}$. In addition to \eqref{compared} and \eqref{compa},  we also need a concentration result similar to \cite[Lemma 8.4]{JL24}. To prove it, we require an overshoot estimate which guarantees that $W^{\beta}_{\tau_u}$ is much smaller than $uK$ with large probability, which has been proved in \cite{JL24_2}.

 \begin{proposition}{\cite[Proposition 2.3]{JL24_2}}\label{init}
For any $p\ge 1$ there exists $C_p$ such that, for all $u>1$,\ 
\begin{equation}
 \bbE\left[ (W^{\beta}_{\tau_u})^p \ \middle| \ \tau_u<\infty \right]\le C_p u^{p}.
\end{equation}
 \end{proposition}

Now let us turn to our replacement for \cite[Lemma 8.4]{JL24}, which we derive from the following technical estimate.

 \begin{lemma}\label{qv_unbounded}
  
	 Let $(U_i)_{i\ge 1}$ be a  sequence of positive,  i.i.d.\ random variables with moments of all orders (positive and negative) and $\bbE[U_i]=1$.
  There exists $\varphi\colon\bbR_+\to \bbR_+$ such that
  for every $r\geq0$ and every non-negative sequence $(\alpha_i)_{i\ge 0}$ such that $\sum_{i\ge 1}\alpha_i=1$,
  \begin{equation}
   \bbE\left[  e^{rV}\right]\le e^{\varphi(r) \sum_{i\ge 1}\alpha^2_i},
  \end{equation}
where $V\coloneqq  \log U-\bbE\left[ \log U\right]$ with  $U=\sum_{i\ge 1} \alpha_i U_i$.
 \end{lemma}

 We postpone the proof of Lemma~\ref{qv_unbounded} to the end of this section.
 Applying
Lemma~\ref{qv_unbounded} to the increments of $M_n$, we obtain the following estimate, which is going to yield the desired concentration property.

 \begin{lemma}\label{petilem}
	 Let $\varphi$ be given by Lemma~\ref{qv_unbounded} in the case where $U_i$ has the same law as $e^{\beta\omega_{1,0}-\gl(\beta)}$. For every $v\geq 0$ and $u>1$, we have
  \begin{equation}
 \bbE\left[ e^{v (M_{\tau_{Ku}}-M_{\tau_u})-\varphi(v) I_{(\tau_u,\tau_{Ku}]}} \ind_{\{\tau_{Ku}<\infty\}} \middle| \ \tau_u<\infty \right]\le 1.
\end{equation}
 \end{lemma}

 \begin{proof}
Recalling \eqref{defmuin}, we have
\begin{equation}
  \frac{W^{\beta}_n}{W^{\beta}_{n-1}}= \sum_{x\in \bbZ^d} D\mu^{\beta}_{n-1}(x) e^{\beta\go_{n,x}-\gl(\beta)}.
 \end{equation}
Since $D \mu_{n-1}(x)$ is $\cF_{n-1}$ measurable and $\go_{x,n}$ is independent of $\cF_{n-1}$, we can apply Lemma~\ref{qv_unbounded}
with $(\alpha_i)_{i\ge 1}$ replaced by $(D\mu_{n-1}(x))_{x\in \bbZ^d}$ and $(U_i)_{i\geq 1}$ by $(e^{\beta \go_{x,n}-\gl(\beta)})_{x\in\bbZ^d}$.
This yields
\begin{equation}\label{concentrator2}
 \bbE\left[ e^{v (M_{n}-M_{n-1})} \middle| \cF_{n-1}\right]\le e^{\varphi(v) I_n}.
\end{equation}
Applying this observation at times $\tau_u+{n-1}$ and $\tau_u+n$  yields, with $\cG_n\coloneqq  \cF_{\tau_u+n}$,
\begin{equation}
 \bbE\left[ e^{v (M_{n+\tau_u}-M_{n-1+\tau_u})-\varphi(v) I_{\tau_u+n}} \middle| \cG_{n-1}\right]\le 1.
\end{equation}
 and hence  $(e^{v M_{n+\tau_u}-\varphi(v) I_{(\tau_u,\tau_u+n]}})_{n\ge0}$ is a supermartingale for the filtration $(\cG_n)$.
 Combining the optional stopping Theorem (at time $n\wedge (\tau_{Ku}-\tau_u)$)
and conditional Fatou (to let $n\to \infty$) we obtain
$$  \bbE\left[ e^{v (M_{\tau_{Ku}}-M_{\tau_u})-\varphi(v) I_{(\tau_u,\tau_{Ku}]}}\ind_{\{\tau_{Ku}<\infty\}}\middle |  \cG_0 \right] \le  \bbE\left[  \liminf_{m\to \infty} e^{v (M_{\tau_{Ku}\wedge m}-M_{\tau_u})-\varphi(v) I_{(\tau_u,\tau_{Ku}\wedge m]}}\middle |  \cG_0 \right]\le 1 $$
which yields the result after  taking expectation over  the $\cG_0$-measurable event $\{\tau_u<\infty\}$.
\end{proof}

We have now all the ingredients to prove our proposition.
\begin{proof}[Proof of Proposition~\ref{keyprop1dis}]
 For this proof only, we set $ \overline \bbP_u\coloneqq  \bbP( \ \cdot \ | \ \tau_u<\infty)$. From Lemma~\ref{dull}, it is sufficient to prove that
 \begin{equation}\label{reduct}
 \overline \bbP_u(  \tau_{Ku}<\infty \ ; \  I_{(\tau_u,\tau_{Ku}]}\le \zeta \log K  )\le K^{-2}.
 \end{equation}
 We have
 \begin{multline}
\overline \bbP_u\left(  \tau_{Ku}<\infty\ ; \  I_{(\tau_u,\tau_{Ku}]}\le \zeta \log K  \right)\\
  \le  \overline \bbP_u\left( W^{\beta}_{\tau_u}> \sqrt{K}u \right)+
     \overline \bbP_u\left(  \tau_{Ku}<\infty\ ; \ I_{(\tau_u,\tau_{Ku}]}\le \zeta \log K ; W^{\beta}_{\tau_u}\le \sqrt{K}u \right).
 \end{multline}
 We are going to prove that both terms in the right hand side are smaller than $ K^{-2}/2$.
Using Markov's inequality and Proposition~\ref{init} , there exists a constant $C>0$ such that
\begin{equation}\label{L6}
    \overline \bbP_u\left( W^{\beta}_{\tau_u}> \sqrt{K}u\right)\le \frac{ \overline \bbE_u[(W_{\tau_u})^6]}{K^3 u^6} \le CK^{-3},
\end{equation}
yielding the right result if $K_0>2C$.
For the second term, if  $\tau_{Ku}<\infty$ and $W^{\beta}_{\tau_u}\le \sqrt{K}u$ then \eqref{compared} implies that $M_{\tau_{Ku}}-M_{\tau_u}\ge \frac{1}{2}\log K$. Hence Lemma~\ref{petilem} for $v=6$ implies that
\begin{multline}
 1\ge  \overline \bbE_u\left[ e^{6 (M_{\tau_{Ku}}-M_{\tau_u})-\varphi(6) I_{(\tau_u,\tau_{Ku}]}}\ind_{ \{  \tau_{Ku}<\infty\ ; \ W^{\beta}_{\tau_u}\le \sqrt{K}u\ ;\ I_{(\tau_u,\tau_{Ku}]}\le \zeta \log K  \}}\right]  \\\ge K^{3-\varphi(6)\zeta}  \  \overline \bbP_u\left(  \tau_{Ku}<\infty\ ; \  I_{(\tau_u,\tau_{Ku}]}\le \zeta \log K\ ; \ W^{\beta}_{\tau_u}\le \sqrt{K}u\right) .
\end{multline}
Setting $\zeta= \frac 1 {2\varphi(6)}$, this implies that, for $K>4$,
$$ \overline \bbP_u\left(  \tau_{Ku}<\infty\ ; \  I_{(\tau_u,\tau_{Ku}]}\le \zeta \log K ; W^{\beta}_{\tau_u}\le \sqrt{K}u \right) \le K^{-5/2}\le K^{-2}/2 $$ and thus the desired result.
\end{proof}

It remains to prove Lemma~\ref{qv_unbounded}. We  rely on an estimate for the two first moments of  $\log U$ (as defined in Lemma \ref{qv_unbounded}) proved in  \cite{CSY03}. More precisely, in  \cite{CSY03} the result is stated in the case of convex combination of  finitely many variables but the proof extends immediately to countable sums by passing to the limit. Let us display it here.

\begin{lemma}{ \cite[Lemma 3.1]{CSY03}}\label{crom}
	Let $(U_i)_{i\ge 1}$ be a sequence of positive i.i.d.\ random variables such that $\bbE[U_1^3+(\log U_1)^2]<\infty$ and $\bbE[U_i]=1$.  There exist $c,C>0$ such that, for any non-negative sequence $(\alpha_i)_{i\ge 0}$ such that $\sum_{i\ge 1}\alpha_i=1$ and setting $U\coloneqq \sum_{i\ge 1}\alpha_i U_i$,
 \begin{equation}\label{yas}
    c \sum_{i\ge 1}\alpha^2_i \le \bbE[\log (1/U)]\le C \sum_{i\ge 1}\alpha^2_i,
 \end{equation}
 and
 \begin{equation}\label{yas2}
        \bbE[(\log U)^2]\le  C \sum_{i\ge 1}\alpha^2_i.
 \end{equation}

\end{lemma}

\begin{proof}[Proof of Lemma~\ref{qv_unbounded}]
Without loss of generality we may assume that $r\ge 1 /4$.
Setting $\bar V= \log U$, we have
\begin{equation}
    \bbE\left[  e^{rV}\right] =     \bbE\left[  e^{r\bar V}\right]\times e^{-r\bbE[\log U]}
\end{equation}
and we use \eqref{yas} to control the second term on the right hand side. For the first term, note that Taylor's formula implies that  ($x_+:= \max(x,0)$ denotes the positive part of $x$)
  \begin{equation}
 \bbE\left[ e^{r \bar V}\right] \le\bbE\left[ 1+ r \bar V +\frac{r^2}{2}(\bar V)^2 e^{r \bar V_{ +}}\right]\le 1+ \bbE\left[ \frac{r^2}{2}(\bar V)^2 e^{r \bar V_{ +}}\right].
 \end{equation}
 Hence, using the inequality $1+x\le e^x$, we can conclude if we can define $\varphi$ which does not depend on the $\alpha_i$s and satisfies
 \begin{equation}
 \bbE\left[ \frac{r^2}{2}\bar V^2 e^{r \bar V_{ +}}\right]\le \varphi(r)\sum_{i\ge 1}\alpha^2_i.
\end{equation}
 We have
 \begin{equation}
  \bbE\left[ \frac{r^2}{2}\bar V^2 e^{r \bar V_{ +}}\right]\le   e^{r}\frac{r^2}{2}\bbE\left[ \bar V^2  \ind_{\{\bar V\le 1\}}\right]+ \bbE\left[\ind_{\{  \bar V> 1\}} e^{2r\bar V} \right].
 \end{equation}
To bound the first term  we simply observe that by \eqref{yas2} we have 
\begin{equation}
 \bbE\left[ \bar V^2 \ind_{\{\bar{V}\le 1\}}\right]\le  \bbE\left[ (\log U)^2\right]\le C \sum_{i\ge 1}\alpha^2_i.
\end{equation}

As for the second one, we have
\begin{equation}
	\bbE\left[\ind_{\{ \bar V\ge 1\}} e^{2r\bar V} \right]\le \bbP\left( \bar V\ge 1\right)^{ 1/2}\bbE\left[e^{4r\bar V} \right]^{ 1/2}.
\end{equation}
 To bound the second factor in a way that does not depend on the $\alpha_i$s, we observe that by Jensen's inequality (recall that we assume $4r\ge 1$) we have
$$\bbE\left[e^{4r\bar V} \right]= \bbE[ U^{4r}]\le 
                                                   \bbE[U^{4{r}}_1].
$$ To conclude the proof  it is sufficient to show that 
\begin{equation}\label{lastdrop}
\bbP\left( \bar V\ge 1\right)\le C \sum_{i\ge 1} \alpha^4_i,
\end{equation}
and to observe that $(\sum_{i\ge 1} \alpha^4_i)^{1/2}\le \sum_{i\ge 1} \alpha^2_i$ by subadditivity.
We have
\begin{equation}
 \bbP\left( \bar V\ge 1\right)= \bbP\left[ U\ge e \right]\le \bbP\left[ \exists i\colon U_i\ge \alpha^{-1/2}_i \right]
 +\bbP\left[ \sum_{i\ge 1} \alpha_i U_i\ind_{\{ U_i\le \alpha^{-1/2}_i\}}\ge e \right].
\end{equation}
The  first term can be bounded by  $
 \bbE\left[ U_1^{8}\right]\sum_{i\ge 1} \alpha^{4}_i$ using union bound and Markov's inequality for $U^{ 8}_i$. The second term can be controlled using Chernov's inequality: For any $a>0$,
 \begin{align}
  &\bbP\left[ \sum_{i\ge 1} \alpha_i U_i\ind_{\{ U_i\le \alpha^{-1/2}_i\}}\ge e \right]\\
  &\le \bbP\left[ \sum_{i} \alpha_i (U_i-1)\ind_{\{ U_i\le \alpha^{-1/2}_i\}}\ge e -1\right]\\
  &\le \bbE\left[ \exp\left(a \left(\sum_{i\ge 1} \alpha_i (U_i-1) \ind_{\{ U_i\le \alpha^{-1/2}_{i}\}}-(e-1)\right)\right)\right].
 \end{align}
Using it with $a=(\alpha_{\max})^{-1/2}$ we obtain (using the inequality $e^x\le 1+x+x^2$, valid for $|x|\le 1$)
\begin{multline}
 \bbE\left[ \exp\left(\alpha_{\max}^{-1/2}  \alpha_i (U_i-1) \ind_{\{ U_i\le \alpha^{-1/2}_{i}\}}\right)\right]\\
 \le 1+ \bbE\left[ \alpha_{\max}^{-1/2}  \alpha_i (U_1-1) \ind_{\{ U_1\le \alpha^{-1/2}_{i}\}}\right]
 + \bbE\left[ \alpha_{\max}^{-1}  \alpha^2_i (U_1-1)^2\right]\\
 \le 1+ C\alpha_{\max}^{-1}  \alpha^2_i \le 1+ C \alpha_i\  \le e^{C\alpha_i}.
\end{multline}
The second term in the second line above is non-positive due to the FKG inequality. Multiplying over all $i\ge 1$, we obtain
\begin{equation}
\bbP\left[ \sum_{i} \alpha_i U_i\ind_{\{ U_i\le \alpha^{-1/2}_i\}}\ge e \right]\le \exp\left( C -(\alpha_{\max})^{-1/2}  (e-1)\right) \le  \frac{ e^{C} \times 8!}{(e-1)^4}  \alpha^4_{\max}
\end{equation}
hence concluding the proof of \eqref{lastdrop} (in the last inequality we use $e^{-\sqrt x}\leq 8!\ x^{-4}$ for $x\geq 0$).
\end{proof}

 \subsection{Proof of Proposition~\ref{keyprop2dis}}\label{dsketch}

Our proof closely follows the one presented in \cite{JL24}. However, since some technical adjustments are necessary,
we provide a sketch of it. In this section we are going to condition on the event $\{\tau_u<\infty\}$
and consider the conditional probability $\bbP_u\coloneqq  \bbP\left[ \cdot    \ | \ \cF_{\tau_u}\right]$. 
We let $\cT=\cT(u,\zeta,K)\coloneqq  \inf\{m\ge \tau_u \ : I_{(\tau_u,m]}\ge \zeta \log K \}$ and we are going to prove that if $\beta>\beta_2$ then there exists a $\delta>0$ such that
 \begin{equation}\label{aaprouver}
  \bbP_u \left( \cT<\infty \  ;    \max_{n\in(\tau_u,\cT]}
  I_n\le \delta  \right)\le  K^{-2}.
\end{equation}
Since $\{I_{(\tau_u,\tau_{Ku}]}\ge \zeta \log K\}\subset\{\cT<\infty\}$, the bound \eqref{aaprouver} directly implies  that
\begin{equation}\label{flows}
\bbP_u \left(   \max_{n\in(\tau_u,\tau_{Ku}]}
  I_n\le \delta \ ; \    I_{(\tau_u,\tau_{Ku}]}\ge \zeta \log K \right)\le  K^{-2}
  \end{equation}
  and we obtain Proposition~\ref{keyprop2dis} by taking  expectation in \eqref{flows} on the event $\{\tau_u<\infty\}$, which has probability smaller than $1/u$ (cf. Lemma~\ref{dull}).
  
\medskip

Let us now expose the main idea we use to prove  \eqref{aaprouver}.   Recalling from \eqref{definechi} that $\chi(\beta)\coloneqq e^{\gl(2\beta)-2\gl(\beta)}-1$, we introduce a parameter $\gep>0$ defined by (recall \eqref{defbeta2} and the assumption $\beta>0$)
\begin{equation}\label{defgep}
  \chi(\beta)\sum_{n=1}^{\infty}P^{\otimes 2} (X^{(1)}_n=X^{(2)}_{n})=:1+4\gep.
\end{equation}
We show that  $I_n\le \delta$, for $\delta$ sufficiently small, implies that a certain bounded stochastic  process $(J_n)$  --  which is obtained  by considering a positive quadratic form based on the Green function evaluated at  $\mu_n$  and is defined below -- has a ``positive drift'' proportional to $I_n$ in the sense that
\begin{equation}\label{posidrift}
I_n \le \delta  \quad \Longrightarrow  \quad \bbE[J_n-J_{n-1}  \ | \ \cF_{n-1}]\ge  2\gep I_n.
\end{equation}
The above implies that on the event $\{\cT<\infty  \ ;  \  \max_{n\in(\tau_u,\cT]}
  I_n\le \delta \}$, the accumulated drift of $J$ on the interval $(\tau_u,\cT]$ is large (at least $2\gep\zeta \log K$) and since $J$ is bounded, it must necessarily be compensated by the martingale part of $J$ taking a large negative value. We then show that the latter is unlikely using martingale concentration estimates.

\medskip

Let us now introduce the functional $J$.
Recalling \eqref{defgep}, we fix  $n_0>0$ such that
 \begin{equation}\label{ennezero}
 \chi(\beta)\sum_{n=1}^{n_0}P^{\otimes 2} (X^{(1)}_n=X^{(2)}_{n})\ge 1+3\gep
 \end{equation}
and introduce $G_0$, the  Green function associated with  $X^{(1)}-X^{(2)}$ truncated at time $n_0$,
\begin{equation}\begin{split}\label{defg0}
 g_{0}(x)&\coloneqq  \sum_{n=1}^{n_0}P^{\otimes 2} (X^{(1)}_n-X^{(2)}_{n}=x),\\
 G_{0}(x,y)&\coloneqq g_{0}(y-x).
\end{split}\end{equation}
The transition matrix of $X^{(1)}-X^{(2)}$ is given by $T\coloneqq D D^{*}=D^* D$ (note that since we are only considering convolution operators on $\bbZ^d$ they all commute)
and  $G_0=\frac{T-T^{(n_0+1)}}{1-T}$.
Next, we define
\begin{equation}\label{deffjn}
 J_n\coloneqq  \left( \mu_{n} , G_0 \mu_{n}\right)=
 \frac{\left( \hat W^{\beta}_{n} , G_0 \hat W^{\beta}_{n}\right)}{ (W^{\beta}_n)^2}.
  \end{equation}
  where, if $a$ and $b$ are  such that $\sum_{x\in \bbZ^d} |a(x)||b(x)|<\infty$, we use the nation
  $$(a,b)=\sum_{x\in \bbZ^d} a(x)b(x).$$
At this point, it is maybe difficult to see the intuition for choosing $J_n$ in this way and we refer to the discussion following \cite[Proposition~7.7]{JL24} and in \cite[Section~7.5]{JL24}.

\smallskip  The truncation to level $n_0$ in the definition of $g_0$ is necessary to treat the case when the random walk $X^{(1)}-X^{(2)}$ is recurrent, but even in the transient case there is a technical reason why we prefer to consider  $G_0$ rather than $G:= T/(1-T)$.
 This  
is because it is convenient to have $\|g_0\|_1<\infty$ . Indeed, using first the Cauchy-Schwarz inequality, then $\|g\ast f\|_2\le \|g\|_1 \|f\|_2$ and finally $\|g_0\|_{1}=n_0$, we have
 \begin{equation}\label{borne}
 0\le J_n\le \|G_0 \mu^{\beta}_{n} \|_2 \| \mu^{\beta}_{n} \|_2\le \|g_0\|_1 \|\mu^{\beta}_{n} \|^2_2=  n_0 \|\mu^{\beta}_{n} \|^2_2.
\end{equation}
 While the above inequality is not required to establish that $J_n$ is bounded (we have $J_n\le \|g_0\|_{\infty}$), it plays a crucial role in the proof of Lemma \ref{Nquadratic}.
 We consider Doob's decomposition of $J_n$,
 \begin{equation}
  J_n-J_0= N_n+ A_n,
 \end{equation}
where $A_n$ and $N_n$ are defined by $A_0=0$, $N_0=0$ and
\begin{equation}\begin{split}
A_{n+1}&=A_n+ \bbE\left[ J_{n+1}-J_n \ | \ \cF_n\right], \\
N_{n+1}&=N_n+ (J_{n+1}-J_n)-\bbE\left[ J_{n+1}-J_n \ | \ \cF_n\right].
\end{split}
\end{equation}
The implication \eqref{posidrift} is proved using the following lower bound on the increments of $(A_n)$ (recall that $\chi(\beta) g_0(0)-1>0$ by construction)
\begin{lemma}\label{aincrement}
 We have
 \begin{equation}\label{rough}
A_{n}-A_{n-1} \ge  \left(\chi(\beta) g_0(0)-1\right)I_n-4\chi(\beta)\left(\left(D \mu^{\beta}_{n-1}\right)^2, G_0\mu^{\beta}_{n-1}\right)
-2\chi_3(\beta) \sum_{x\in \bbZ^d }D \mu^{\beta}_{n-1}(x)^3,
 \end{equation}
where $\chi_3(\beta)\coloneqq \bbE\left[ (e^{\beta \go_{1,0}-\gl(\beta)}-1)^3\right]= e^{\gl(3\beta)-3\gl(\beta)}-3 e^{\gl(2\beta)-2\gl(\beta)}+2$.
As a consequence, there exists a constant $C$ (which depends on $\beta$) such that
 \begin{equation}\label{lafete}
A_{n}-A_{n-1} \ge  \left(\chi(\beta) g_0(0)-1\right)I_n-CI^{3/2}_n.
 \end{equation}

\end{lemma}

Since $J_n\ge 0$, the above bound can be used to control the increments of $(N_n)$. Next, we obtain the following.
\begin{lemma}\label{Nquadratic}
We have, almost surely for every $n$,
\begin{equation}
\bbE\left[ (N_{n}-N_{n-1})^2 \ | \ \cF_{n-1}\right]\le \kappa I^2_n
\end{equation}
where (recall \eqref{defnu}) $\kappa\coloneqq  n^2_0 \left( \| D \|_{\infty}^{-4}+ e^{\frac{\gl(8\gb)+\gl(-8\gb)}{2}}\right)$.
\end{lemma}
The proofs of Lemmas~\ref{aincrement} and~\ref{Nquadratic} replicate those of the analogous statements proved for the simple random walk \cite[Lemma 8.5 and Lemma 8.6]{JL24}. At the end of the section we provide indications about the changes that are required.
The last input needed for the proof of Proposition~\ref{keyprop2dis} is a concentration result for bounded martingales from \cite{JL24}.
\begin{lemma}\cite[Lemma~8.4]{JL24}  \label{leAU}
 Let $(N_n)$ be a discrete-time martingale starting at $0$, with increments that are bounded in absolute value by $A > 0$. For $v > 0$, let $T_v$ be the first time that $(N_n)$ hits $[v,\infty)$.
For any $a > 0$, we have
$$P\left[\langle N\rangle_{T_v} \le  a \ ; \ T_v<\infty \right] \le e^{-\frac{v}{A+1}\left(\log \left(\frac{v}{a(A+1)}\right)-1\right)}.$$
 As a consequence, for any stopping time $T$ we have 
$$ P\left[\langle N\rangle_{T} \le  a \ ; \ T<\infty \ ; \ N_T\ge v \right] \le e^{-\frac{v}{A+1}\left(\log \left(\frac{v}{a(A+1)}\right)-1\right)}.$$

\end{lemma}

\begin{proof}[Proof of \eqref{aaprouver} assuming Lemmas~\ref{aincrement} and~\ref{Nquadratic}]

If $I_n\le \delta$, recalling \eqref{ennezero} we have, from \eqref{lafete},
\begin{equation}
 A_n-A_{n-1}\ge I_n \left(\chi(\beta) g_0(0)-1- C\delta^{1/2}\right)
 \ge 2\gep I_n
\end{equation}
where the second inequality is valid for $\delta\le \delta_0(\gep)$ sufficiently small.
Thus,  if $\cT<\infty$ and $\max_{n\in (\tau_u,\cT]} I_n\le \delta$, we have
\begin{equation}\label{controla}\begin{split}
 A_{\cT}-A_{\tau_u}\ge 2\zeta \gep \log K.
 \end{split}
\end{equation}
Using \eqref{borne} and the trivial bound  $\|\mu^{\beta}_n\|_2\le 1$, we have
\begin{equation}\label{controln}
 N_{\cT}-N_{\tau_u}=A_{\tau_u}-A_{\cT}+J_{\cT}-J_{\tau_u}
 \le A_{\tau_u}-A_{\cT}+n_0.
\end{equation}
Hence \eqref{controln} and \eqref{controla} imply that, for $K\ge K_0(n_0,\zeta,\gep)$ sufficiently large,
\begin{equation}\label{bn}
  N_{\cT}-N_{\tau_u}\le -  \zeta\gep\log K.
\end{equation}
On the other hand, if $\max_{n\in (\tau_u,\cT]} I_n\le \delta$, we have, from Lemma~\ref{Nquadratic} and the definition of $\cT$,
\begin{equation}\label{bquadra}
 \langle N\rangle_{\cT}-\langle N\rangle_{\tau_u}\le \kappa \sum_{n=\tau_u+1}^{\cT} I^2_n\le  \kappa \left(\delta \sum_{n=\tau_u+1}^{\cT-1} I_n+\delta^2\right) \le \kappa\delta ( \zeta \log K+\delta )\le 2\kappa\delta\zeta \log K.
\end{equation}
Thus we obtain
\begin{multline}\label{above_event}
  \bbP_u\left(  \max_{n\in (\tau_u,\cT]}
  I_n\le \delta\ ; \ \cT<\infty \right)  \\  \le \bbP_u\left
(  \cT<\infty \ ; \ N_{\cT}-N_{\tau_u}\le -  \zeta\gep\log K \ ; \    \langle N\rangle_{\cT}-\langle N\rangle_{\tau_u}  \le2\kappa\delta\zeta \log K \right).
\end{multline}
Using Lemma~\ref{leAU} for the martingale $(N_{\tau_u}-N_{\tau_u+m})_{m\ge 0}$, stopping time $\cT-\tau_u$ with $A=n_0$ (which is a bound for the increments of $N$ cf.\ \eqref{borne}), $v=\zeta\gep\log K$ and  $a=2\kappa\zeta \delta \log K$) we have
\begin{equation}
\bbP_u\left
( N_{\tau_u}- N_{\cT}\ge v \ ; \cT<\infty  \ ; \     \langle N\rangle_{\cT}-\langle N\rangle_{\tau_u}  \le a \right) \le e^{-\frac{v}{n_0+1}(\log(\frac{v}{a(n_0+1)})-1)}.
\end{equation}
To conclude the proof of \eqref{aaprouver} and hence of Proposition~\ref{keyprop2dis}, we need to check that the exponent is smaller than $-2\log K$,
which is immediate if one chooses $\delta$  sufficiently small (since $\gep$ is fixed).
\end{proof}

 \begin{proof}[Proof of Lemma~\ref{aincrement}]
 
 Let us first explain how \eqref{lafete} is deduced from \eqref{rough}. We have
\begin{equation}\begin{split}
 \sum_{x\in \bbZ^d }D \mu^{\beta}_{n-1}(x)^3&\le I_n\| D\mu^{\beta}_{n-1}\|_{\infty}\le I^{3/2}_n,\\
\left(\left(D \mu^{\beta}_{n-1}\right)^2, G_0\mu^{\beta}_{n-1}\right)
&\le I_n\| G_0\mu^{\beta}_{n-1}\|_{\infty}\le  I_n \|g_0\|_1 \|\mu^{\beta}_{n-1}\|_{\infty}\le \frac{n_0}{\|D\delta_{0} \|_{\infty}} I^{3/2}_n.
\end{split}\end{equation}
where in the last inequality above we used the fact that $\|\mu^{\beta}_{n-1}\|_{\infty}\le \frac{\| D\mu^{\beta}_{n-1}\|_{\infty}}{\|D\delta_0 \|_{\infty}}$. These bounds replace those given in the display following \cite[display (61)]{JL24}.
The proof of \eqref{rough} is now identical to the one presented in the simple random walk case, namely \cite[Lemma 8.5]{JL24}.
The only modification needed is to replace $D^2$ by $T=DD^*$ in \cite[Equation (67)]{JL24}.
\end{proof}

\begin{proof}[Proof of Lemma~\ref{Nquadratic}]
 Using  \eqref{borne}, we have
\begin{equation}\begin{split}
 \bbE\left[ (N_{n}-N_{n-1})^2 \ | \ \cF_{n-1}\right]&= \bbE\left[ (J_{n}-J_{n-1})^2 \ | \ \cF_{n-1}  \right]-\bbE\left[J_{n}-J_{n-1} \ | \ \cF_{n-1} \right]^2\\ & \le  \bbE\left[ J^2_{n}+J^2_{n-1} \ | \ \cF_{n-1}  \right] \\
 &\le n^2_0\left(\bbE\left[ \|\mu^{\beta}_{n}\|^4_2   \ | \ \cF_{n-1}  \right]+  \|\mu^{\beta}_{n-1}\|^4_2\right).
 \end{split}
 \end{equation}
 To bound the second term we observe that  $(D\mu)(x)\geq \|D\|_\infty \mu(x+z_{\max})$, where $z_{\max}\in\bbZ^d$ satisfies $D(0,z_{\max})=\|D\|_\infty$, and thus $I_n=\|D\mu_{n-1}^\beta\|_2^2\geq \| D \|^2_{\infty}\|\delta_{z_{\max}}*\mu_{n-1}^\beta\|_2^2=\| D \|^2_{\infty}\|\mu_{n-1}^\beta\|_2^2$. To conclude, we need to prove that
\begin{equation}\label{anotherineq}
 \bbE\left[ \|\mu^{\beta}_{n}\|^4_2 \ | \ \cF_{n-1}\right]\le e^{\frac{\gl(8\gb)+\gl(-8\gb)}{2}} I^2_n,
\end{equation}
 which can be done by replicating the argument in the proof \cite[Lemma 8.6]{JL24}.
\end{proof}

\section{Proof of Theorem~\ref{strongbeta}}\label{sbproof}

\subsection{Proof strategy}
Let us assume that $\beta_2=0$ and fix $\beta>0$ such that $\gl(2\beta)<\infty$. We are going to identify a sequence of events $A_n$ which is such that 
\begin{equation}\label{lelimitz}
\lim_{n\to \infty}\bbP(A_n)= 0 \quad \text{ and } \lim_{n\to \infty} \tilde \bbP_n(A^{\complement}_n)=0.
\end{equation}
From Lemma~\ref{babac}, this implies that $\lim_{n\to \infty} \bbE[(W^{\beta}_n)^{\theta}]=0$ for $\theta\in[\frac 12,1)$ and hence $W^{\beta}_n$ converges to zero in probability. More precisely, we are going to prove that \eqref{lelimitz} holds along a diverging sequence of integers $n$ (which is of course sufficient for our purpose).

\medskip

The key point is to find an observable $R_n$ (a function of $\go$) whose value is greatly affected by the size biasing. Aiming for the simplest possible choice, we take a linear combination of the $(\go_{k,x})$.
We set $p_k(x)\coloneqq P(X_k=x)$ and define
\begin{equation}
R_n\coloneqq \sum_{k=1}^n \sum_{x\in \bbZ^d} p_k(x)\go_{k,x}.
\end{equation}
The choice of the coefficients  is made with the spine representation in mind (recall Section~\ref{spineR}), since $p_k(x)$ is the probability that the site $(k,x)$ is visited by the spine.
By \eqref{standard}, we have $\bbE[R_n]=0$ and
\begin{equation}\label{lavariance}
\bbE\left[ R^2_n\right]=\sum_{k=1}^n  \sum_{x\in \bbZ^d} p_k(x)^2\eqqcolon \Sigma_n,
\end{equation}
On the other hand, from Lemma~\ref{sbrepresent} we have
\begin{equation}\label{lamean}
 \tilde\bbE_n\left[ R_n\right]= \hat \bbE\otimes E\left[ \sum_{k=1}^n p_k(X_k)\hat \go_k \right]=\gl'(\beta) \Sigma_n,
\end{equation}
 where in the last identity we used that $\hat \bbE[\hat \go_k]= \gl'(\beta)$ (which is an immediate consequence of the definition \eqref{tilted}).
Recalling \eqref{lequiv},
since $\beta_2=0$ and $\Sigma_n$ is the expected number of return to zero before time $n$ of $(X^{(1)}_k-X^{(2)}_k)_{k\ge 0}$, we have $\lim_{n\to \infty} \Sigma_n=\infty$ .
Now, \eqref{lavariance} implies  that the typical order of magnitude of $R_n$ under $\bbP$ is at most  $\sqrt{\Sigma_n}$. On the other hand,
\eqref{lamean} indicates that $R_n$ may be typically much larger -- of order $\Sigma_n$ -- under $\tilde \bbP_n$.
This motivates our definition for the event 
\begin{equation}\label{defan}
 A_n\coloneqq \left\{ R_n\ge \Sigma_n^{3/4} \right\}.
\end{equation}

We are going to prove the following 
\begin{proposition}\label{estimatean}
For any  $\beta>0$ such that $\gl(2\beta)<\infty$ and $A_n$ defined as in \eqref{defan} we have
\begin{equation}
 \lim_{n\to \infty} \bbP(A_n)=0 \quad \text{ and } \quad \liminf_{n\to \infty} \tilde\bbP_n(A^{\complement}_n)=0.
\end{equation}
\end{proposition}
A consequence of the above and Lemma~\ref{babac} is that $\liminf_{n\to \infty} \bbE\left[(W^{\beta}_n)^{\theta}\right]=0$ for $\theta\in[\frac 12,1)$ and thus that strong disorder holds, proving Proposition~\ref{strongbeta}.

\subsection{Proof of Proposition~\ref{estimatean}}

The first half of the statement directly follows from \eqref{lavariance}, since $\lim_{n\to\infty}\Sigma_n=\infty$ and
\begin{equation}
\bbP(A_n)\le  \bbE[R^2_n] (\Sigma_n)^{-3/2}=(\Sigma_n)^{-1/2}.
\end{equation}
For the second half of the statement we use Lemma~\ref{babac}. We have
\begin{equation}
 \tilde \bbP_n  (A^{\complement}_n)=\bbP\otimes \hat \bbP\otimes P\left( \tilde \go( \go,\hat \go,X)\in A^{\complement}_n\right).
\end{equation}
To estimate the above, we split $R_n(\tilde \go)$ into two parts
\begin{equation}
 R_n(\tilde \go)=\gl'(\beta)\sum_{k=1}^n p_k(X_k)     +  \sum_{k=1}^n \sum_{x\in \bbZ^d}  p_k(x)\left(\tilde \go_{k,x}-\gl'(\beta)\ind_{\{X_k=x\}} \right)=:R^{(1)}_n(X)+ R^{(2)}_n(\go,\hat \go,X).
\end{equation}
Hence we have
\begin{equation}
  \bbP\otimes \hat \bbP\otimes P( R_n(\tilde \go)< \Sigma^{3/4}_n)
  \le P\left(R^{(2)}_n(\go,\hat \go,X)\le -\Sigma^{3/4}_n\right)+\bbP\otimes \hat \bbP\otimes P\left( R^{(1)}_n(X)\le 2\Sigma^{3/4}_n\right).
\end{equation}
We are going to show that along some subsequence, both terms on the right hand side converge to zero.
Starting with $R^{(2)}_n(\go,\hat \go,S)$, we have for every realization of $X$
\begin{equation}
 \bbE\otimes \hat \bbE\left[ R^{(2)}_n(\go,\hat \go,X)^2\right]\leq \Sigma_n+\left(\gl''(\beta)-1\right)\sum_{k=1}^n p_k(X_k)^2\le \max(1, \gl''(\beta)) \Sigma_n
\end{equation}
(we have used that $\hat \bbE\left[ (\go-\gl'(\beta))^2\right]=\gl''(\beta)$) so that by Chebychev's inequality
\begin{equation}
  \bbP\otimes \hat \bbP\otimes P\left( R^{(2)}_n(\go,\hat \go,X) \le - \Sigma_n^{3/4}\right)\le
  \max(1, \gl''(\beta))\Sigma_n^{-1/2}.
\end{equation}
Hence we have the desired convergence to zero.
To conclude, we need to show that
\begin{equation}\label{r1converge}
\liminf_{n\to \infty}    P\left( R^{(1)}_n(X) \le 2\Sigma_n^{3/4}\right)=0.
\end{equation}
 Contrary to $R^{(2)}_n$, the variable $R^{(1)}_n$ is not -- in full generality -- concentrated around its mean.  This is the  reason why we do not employ the second moment method to prove \eqref{r1converge}.
To add symmetry to the problem, it is convenient to consider, rather than $R^{(1)}_n(X)$, a random variable with an extra layer of randomness whose conditional mean is equal to $R^{(1)}_n(X)$. More specifically, we let $X'$ denote a simple random walk which is independent of $X$ and has the same distribution (we let $P\otimes P'$ denote the distribution of $(X,X')$), and set
\begin{equation}
 Q_n(X,X')\coloneqq   \gl'(\beta)\sum_{k=1}^n \ind_{\{X_k=X'_k\}}.
\end{equation}
We have $R^{(1)}_n(X)=E'[ Q_n(X,X')]$ and thus
\begin{equation}\label{stoup}
   P\left( R^{(1)}_n(X) \le 2\Sigma_n^{3/4}\right)\le 2   P\otimes P' \left( Q_n(X,X') \le 4\Sigma_n^{3/4}\right).
\end{equation}
Let us shortly justify \eqref{stoup}. If $f(X,X')$ is a nonnegative function and $u\ge 0$, we have by Markov inequality
$$ E'[f(X,X')]\le u \quad \implies  \quad P'(f(X,X')\le 2u)\ge \frac{1}{2},$$
and thus, using Markov inequality again for the variable $Y= P'(f(X,X')\le 2u)$,
$$ P(E'[f(X,X')]\le u  )\le P\left(  Y\ge \frac{1}{2}\right)\le 2E[ Y] = 2 P\otimes P' \left( f(X,X') \le 2u\right).$$
 The remaining step is to bound the right-hand side in \eqref{stoup}.
 For notational simplicity, we set $\bP\coloneqq  P\otimes P'$ and introduce the renewal process
 $$ \tau\coloneqq \{ k\ge 0 \colon X_k=X'_k\}.$$

With this new notation, we have
\begin{equation}
 P\otimes P' \left( Q_n(X,X') \le 4\Sigma_n^{3/4}\right)
 =\bP\left(\gl'(\gb)  |\tau\cap[1,n]| \le 4 \bE\left[ |\tau\cap[1,n]|\right]^{3/4}\right).
\end{equation}
We conclude the proof of \eqref{r1converge} and hence of Proposition~\ref{estimatean} by applying the following technical result (proved in the next subsection) to the renewal $\tau$. \qed
\begin{lemma}\label{renewnew}
Given $c>0$ and $\tau$ a recurrent renewal process, we have
\begin{equation}\label{aprouver}
 \liminf_{n\to \infty}   \bP\left(  |\tau\cap[1,n]| \le c \bE\left[ |\tau\cap[1,n]|\right]^{3/4}\right)=0.
\end{equation}
\end{lemma}

\subsection{Proof of Lemma~\ref{renewnew}}
 With a small abuse of notation, we identify the set $\tau$ with an increasing sequence $(\tau_i)_{i\ge 0}$.
We set
$ \alpha_n\coloneqq  \bE\left[ \tau_1\wedge n\right]$  and $\bar K(n)\coloneqq \bP\left(\tau_1>n\right).$
Setting $\bar \tau_k\coloneqq \sum_{i=1}^k (\tau_{i}-\tau_{i-1})\wedge n$, we have for any $k\ge 1$
\begin{equation}\label{topz}
 \bP\left( |\tau\cap[1,n]| \le k \right)= \bP\left( \tau_{k+1} > n\right)= \bP\left( \bar \tau_{k+1} > n\right)
 \le \frac{\bE\left[\bar \tau_{k+1}\right] }{n}=\frac{(k+1)\alpha_n}{n}.
\end{equation}
On the other hand, we have
$$\bP\left( |\tau\cap[1,n]| \ge i\right) \le \bbP(\forall j\in[1,i]\colon \tau_{i+1}-\tau_i\le n )= (1-\bar K(n))^i$$ which implies after summing over $i$ that
\begin{equation}
 \bE\left[ |\tau\cap[1,n]|\right]\le (\bar K(n))^{-1}\wedge n.
\end{equation}
Replacing $k$ in \eqref{topz} by $c((\bar K(n))^{-1}\wedge n)^{3/4}$, we obtain that \eqref{aprouver} holds
if
\begin{equation}\label{verszero}
 \liminf_{n\to \infty} ( (\bar K(n))^{-1}\wedge n)^{3/4} \frac{\alpha_n}{n}=0.
\end{equation}
If $\liminf_{n\to\infty} \alpha_n n^{-1/4}=0$, then there is nothing to prove.
If this is not the case, then we consider $n$ such that
\begin{equation}\label{nondecreasing}
\alpha_{n+1} (n+1)^{-1/5}\ge \alpha_n n^{-1/5}
\end{equation}
(the fact that $\alpha_n n^{-1/5}$ diverges implies in particular that it is not eventually decreasing so that one can find an infinite sequence satisfying the above).
Since we have $\alpha_{n+1}=\alpha_n+\bar K(n)$, \eqref{nondecreasing} implies that for $n$ sufficiently large
\begin{equation}
 \bar K(n)\ge \left(\left(1+\frac{1}{n}\right)^{1/5}-1\right) \alpha_n\ge \frac{\alpha_n}{6n} 
\end{equation}
and we obtain that, along the subsequence satisfying \eqref{nondecreasing},
\begin{equation}
 ( (\bar K(n))^{-1}\wedge n)^{3/4} \frac{\alpha_n}{n}\le 6 (\alpha_n/n)^{1/4}.
\end{equation}
Since $\alpha_n=o(n)$ by dominated convergence, this is sufficient to conclude that \eqref{verszero} holds.
\qed

\appendix

\section{Proof of Proposition~\ref{alphaundemi}}\label{proofalpha}
We assume that $\beta_2>0$ and $\alpha>1/2$ (recall \ref{defalpha}).
We are going to  show that there exist $\beta>\beta_2$ and  $\gamma\in(0,1)$ such that
\begin{equation}\label{uiproof}
\sup_{n\ge 0}\bbE\left[ (W^\beta_n)^{1+\gamma}\right]=\sup_{n\ge 0}\tilde \bbE_n\left[ (W^\beta_n)^{\gamma}\right]<\infty.
\end{equation}
This implies uniform integrability of $W^{\beta}_n$, hence that  weak disorder holds at $\beta>\beta_2$.
 The first step of the proof, inspired from \cite[Section 1.4]{BS10}, is to reduce the problem to the control of the partition function of another model, \textit{the disordered pinning model}. The argument is based on the size-biased representation from Lemma~\ref{sbrepresent}. Letting $\tilde W^{\beta}_n\coloneqq  E'\Big[ e^{\sum^n_{k=1} (\beta \tilde \go_{k,X'_k}-\gl(\beta))}\Big]$ (where $X'$ with law $P'$ has the same distribution as $X$)
we observe that 
\begin{equation}\label{jensons}
 \tilde \bbE_n\left[ (W^\beta_n)^{\gamma}\right]= \hat\bbE\otimes \bbE\otimes E \left[ (\tilde W^{\beta}_n)^{\gamma}\right]
 \le \hat\bbE  \left[ \left(\bbE\otimes E \left[\tilde W^{\beta}_n\right]\right)^{\gamma}\right].
\end{equation}
Now, using the notation $\delta_n\coloneqq \ind_{\{X_n=X'_n\}}$ and $\bP=P\otimes P'$, we introduce a notation $(\hat Z_n)$ for the partition function appearing in the right-hand side of \eqref{jensons}, as well as a counterpart  $\hat Z^{c}_n$ with constrained endpoint:  
\begin{equation}\begin{split}\label{lesZ}
 \hat Z_n&\coloneqq \bbE\otimes E[ \tilde W^{\beta}_n]= \bE\left[  e^{\sum^n_{k=1} \left(\beta \hat \go_k-\gl(\beta)\right)\delta_k}\right],\\
  \hat Z^{c}_n&\coloneqq \bE\left[  e^{\sum^n_{k=1} \left(\beta \hat \go_k-\gl(\beta)\right)\delta_k}\delta_n\right],
\end{split}\end{equation}
with the convention $\hat Z_0=  \hat Z^{c}_0=1$.
With \eqref{uiproof} and \eqref{jensons}, the proof of Proposition~\ref{alphaundemi} is reduced to proving that
\begin{equation}\label{lesup}
 \sup_{n\ge 0} \hat \bbE\left[  (\hat Z_n)^{\gamma}\right]<\infty.
\end{equation}
To show the above we rely on a method developed in \cite{DGLT}. The notable differences with \cite{DGLT} are that we are trying to control the unconstrained partition function rather than the constrained one, and that we make no regularity assumptions on the  interarrival law $K(n)\coloneqq \bP(\tau_1=n)$, but these differences only result in minor modifications in the argument.
We refer to \cite{DGLT} for more insight concerning the proof.

\medskip

We fix an integer $m\ge 1$ and set $\gamma=1-\frac 1{\log m}$  and $\beta=\beta_2+\frac{1}{m^2}$. We define the shifted environment $\theta_j \hat \go$ by setting $(\theta_j \hat \go)_k= \hat \go_{j+k}$ and the shifted partition functions $\theta_j\hat Z_{n}$ by replacing $\hat \go$ by $\theta_j \hat \go$ in  \eqref{lesZ}.
Decomposing the partition function according to the value of the last renewal point before $m$ (variable $a$) and first renewal point between $m$ and $n$ (variable $b$), we obtain that
\begin{equation}
 \hat Z_n = \sum_{a=0}^{m-1} \hat Z^c_a \left( \sum_{b=m}^{n}  K(b-a) e^{\beta \hat \go_b-\gl(\beta)} \theta_b\hat Z_{n-b} + \bP(\tau>n-a)  \right).
\end{equation}
Thus, setting $A_n=\hat \bbE\left[ (\hat Z_n)^{\gamma}\right]$ and $B_n=\hat \bbE\left[ (\hat Z^{c}_n)^{\gamma}\right]^{1/\gamma}$ and using the subadditivity of $x\mapsto x^{\gamma}$, we obtain
 \begin{equation}\label{vyvyvy}
 A_n \le  \sum_{a=0}^{m-1}  B^\gamma_a \left( \sum_{b=m}^{n}  K(b-a)^{\gamma} \hat  \bbE\left[ e^{\gamma\beta \hat \go_b-\gamma\gl(\beta)}\right]  A_{n-b} + \bP(\tau>n-a)^{\gamma}\right).
\end{equation}
Let us now set $\bar A_n\coloneqq  \max(A_0,\dots, A_{n})$ and observe that with our choice for $\beta$ and $\gamma$
$$\hat \bbE\left[ e^{\gamma\beta \hat \go_1-\gamma\gl(\beta)}\right]= \bbE\left[ e^{(1+\gamma)\beta \hat \go_1-(1+\gamma)\gl(\beta)}\right] \le \bbE\left[ e^{\gl(2(\beta_2+1))- 2\gl(\beta_2+1)} \right]=:\rho .$$ 
The inequality \eqref{vyvyvy}  implies that
\begin{equation}\begin{split}\label{tropz}
 A_n&\le \rho  \left(\sum_{a=0}^{m-1} \sum_{b=m}^{n}   B^{\gamma}_a K(b-a)^{\gamma} \right)\bar A_{n-m}+\left(\sum_{a=0}^{m-1} B^{\gamma}_a\right)  \\
 &\le  2\rho \left(\sum_{a=0}^{m-1}  \sum_{b=m}^{\infty}   B^{\gamma}_a K(b-a)^{\gamma}\right) \bar A_{n-m}+ m.
\end{split}\end{equation}
To obtain the first line we used that $\bP(\tau>n-a)^{\gamma}\le 1$ and in the second line we used the bound $B^{\gamma}_a\le  2$, which will be proved below in \eqref{notsharp}.
From \eqref{tropz}, we deduce that $(A_n)_{n\in\N}$ is bounded if
\begin{equation}\label{cvrai}
	 2\rho\left(\sum_{a=0}^{m-1} \sum_{b=m}^{\infty}   B^{\gamma}_a K(b-a)^{\gamma} \right)< 1.
\end{equation}
We are going to prove that \eqref{cvrai} holds for $m$ sufficiently large. We do so by combining two arguments. The first is the observation that by taking $\gamma$ sufficiently close to one, we can, at the cost of some small error term, drop the power of $\gamma$ in our sum.

\begin{lemma}\label{stepp1}
 For $\gamma=1-\frac{1}{\log m}$ and $\beta=\beta_2+\frac{1}{m^2}$,
there exists $C>0$ such that for all $m$ sufficiently large  we have
\begin{equation}
\sum_{a=0}^{m-1} \sum_{b=m}^{\infty}   B^{\gamma}_a K(b-a)^{\gamma} \le C \left(\sum_{a=0}^{m-1} \sum_{b=m}^{\infty}  \ B_a K(b-a) \right)+ C m^{-1}.
\end{equation}
\end{lemma}
The second point is to demonstrate, via a change of measure argument, that for most $n$ we have
$B_n \ll \hat \bbE\left[ \hat Z^{c}_n\right]$, and that, as a consequence, the sum we wish to control is small.

\begin{lemma}\label{stepp2}
 
For $\gamma=1-\frac{1}{\log m}$ and $\beta=\beta_2+\frac{1}{m^2}$,  we have
 \begin{equation}\label{aprove}
 \lim_{m\to \infty}\sum_{a=0}^{m-1} \sum_{b=m}^{\infty}    B_a K(b-a)=0.
 \end{equation}
 
\end{lemma}
The combination of Lemma \ref{stepp1} and Lemma \ref{stepp2} allows to conclude that \eqref{cvrai} holds.
\begin{proof}[Proof of Lemma \ref{stepp1}]
 The idea is that $B^{\gamma}_a K(b-a)^{\gamma}\le C B_a K(b-a)$ for the values of $a$ and $b$ that most contribute to the sum and bound the remainder of the contribution by $Cm^{-1}$.
 Letting $\kappa$ be a large integer and using that $\gamma=1-\frac 1{\log m}$, we have
 \begin{multline}
  \sum_{a=0}^{m-1} \sum_{b=m}^{\infty}   B^{\gamma}_a K(b-a)^{\gamma} \le e^{\kappa /\gamma}\sum_{a=0}^{m-1} \sum_{b=m}^{\infty}    B_a K(b-a)\ind_{\{ B_a^{\gamma} K(b-a)^{\gamma} \ge m^{-\kappa}\}}\\
  +   \sum_{a=0}^{m-1} \sum_{b=m}^{\infty}   B^{\gamma}_a K(b-a)^{\gamma} \ind_{\{ B^{\gamma}_a K(b-a)^{\gamma} < m^{-\kappa}\}}.
 \end{multline}
To control the second term, we split it into two. As the  double-sum below contains $m^{\kappa- 1}$ terms, we have
\begin{equation}
 \sum_{a=0}^{m-1} \ \sum_{b=m}^{m+m^{\kappa-2}-1}   B^{\gamma}_a K(b-a)^{\gamma} \ind_{\{ B^{\gamma}_a K(b-a)^{\gamma} < m^{-\kappa}\}}\le m^{-1}.
\end{equation}
On the other hand, using again $B_a\le 1$, we have
\begin{equation}
 \sum_{a=0}^{m-1} \sum_{b> m+m^{\kappa-2}}   B^{\gamma}_a K(b-a)^{\gamma}\le m \sum_{r\ge m^{\kappa-2}}K(r)^{\gamma}.
\end{equation}
Finally, we observe that for $m$ sufficiently large
\begin{equation}\label{astute}
\sum_{r\ge m^{\kappa-2}}K(r)^{\gamma}\le \sum_{r\ge m^{\kappa-2}} \max(r^{\alpha/3}K(r), r^{-2})
\le C \left( m^{-(\kappa-2)\alpha/3}+ m^{2-\kappa}\right)\le m^{-2}.
 \end{equation}
The first inequality above boils down considering separately the cases  $K(r)\ge r^{-3}$ (for which the first bound holds provided that $\log m\ge 9/\alpha$) and $K(r)< r^{-3}$ (for which the second bound holds provided that $m$ is not too small). For the second inequality we recall the definition \eqref{defalpha} of $\alpha$ and observe that grouping terms, we have for $n$ sufficiently large
\begin{equation}
 \sum_{r\ge n} r^{\alpha/3}K(r)\le \sum_{k\ge 1}  (2^k n)^{\alpha/3}\bP\left(\tau_1\in [n2^{k-1},n 2^{k})\right)\le \sum_{k\ge 1}  (2^k n)^{\alpha/3} (2^k n)^{-3\alpha/4}\le  n^{-\alpha/3}.
\end{equation}
The last inequality in \eqref{astute} is valid provided $\kappa$ is chosen sufficiently large (we have not chosen an explicit $\kappa$ to underline that this step of the proof only requires $\alpha>0$).
\end{proof}

\begin{proof}[Proof of Lemma \ref{stepp2}]
Let us set $K''(n)=  e^{\gl(2\beta_2)-2\gl(\beta_2)} K(n)$ (recall that we have $\sum_{n\ge 1} K''(n)=1$ by \eqref{defbeta2}) and let $\tau''$ denote a renewal process with interarrival law $K''$ (let $\bP''$ denote the distribution).
We are going to show that 
\begin{equation}
 \lim_{m\to \infty}\sum_{a=0}^{m-1} \sum_{b=m}^{\infty}    B_a K''(b-a)=0.
\end{equation}
 A first idea might be to note that, by Jensen's inequality, we have
\begin{multline}\label{notsharp}
 B_a\le \hat \bbE\left[ \hat Z^{c}_a\right]= \bE\left[ e^{(\gl(2\beta)-2\gl(\beta))\sum_{k=1}^a \delta_k}\delta_a\right]\\ \le e^{ \frac{Ca}{m^2} }\bE\left[ e^{(\gl(2\beta_2)-2\gl(\beta_2))\sum_{k=1}^a \delta_k}\delta_a\right] \le  2   \bP''(a\in \tau''),
\end{multline}
where in the second inequality we used the fact that with our choice of $\beta$, $(\gl(2\beta)-2\gl(\beta))\le (\gl(2\beta_2)-2\gl(\beta_2))+C m^{-2}$ and the last inequality is valid for $m$ sufficiently large since $a\le m$.
Hence we have 
\begin{equation}
 \sum_{a=0}^{m-1} \sum_{b=m}^{\infty}    B_a K''(b-a)\le  2 \sum^{m-1}_{a=0} \bP''(a\in \tau'')\bP''(\tau''_1\ge m-a)=2
\end{equation}
(the summand in the right-hand side is the probability that $a$ is the last renewal point before $m$).
 Since this is not sufficient for our purposes, the idea in our proof is to use something sharper than Jensen's inequality in \eqref{notsharp}.
 We use H\"older inequality instead, and observe, for any positive function $g(\hat \go)$,
 \begin{equation}\label{holderoncemore}
B_a\le \hat \bbE\left[g(\hat \go)\hat Z^c_a\right] \hat \bbE\left[ g(\hat \go )^{-\frac{\gamma}{1-\gamma}} \right]^{\frac{1-\gamma}{\gamma}}.
 \end{equation}
We set $g(\hat \go)= e^{-\gep_m\sum_{i=1}^m \hat \go_i+m[\gl(\beta)-\gl(\beta-\gep_m)] }$ with $\gep_m\coloneqq  (m\log m)^{-1/2}$.
We have 
\begin{equation}
  \frac{1-\gamma}{\gamma}\log \hat\bbE\left[ g(\hat \go )^{-\frac{\gamma}{1-\gamma}} \right]= m\left( \frac{1-\gamma}{\gamma} \gl\left(\beta+   \frac{\gamma\gep_m}{1-\gamma}\right)+\gl(\beta-\gep_m)-\frac{1}{\gamma}\gl(\beta) \right)
  \le C  \frac{m\gep^2_m}{1-\gamma}.
\end{equation}
The last inequality is Taylor's formula at order $2$ with $C$ being a bound for $\gl''$ in an interval around $\beta$ (say $[0,\beta_2+1]$). With our choice of $\gep_m$, we obtain that ${ \hat\bbE}\left[ g(\hat \go )^{-\frac{\gamma}{1-\gamma}} \right]^{\frac{1-\gamma}{\gamma}}$ is uniformly bounded in $m$.
Summing over $a$ and $b$ in \eqref{holderoncemore}, we can conclude if we can show that 
\begin{equation}\label{lastwist}
\lim_{m\to \infty}  \sum_{a=0}^{m-1} \sum_{b=m}^{\infty}  { \hat\bbE}[g(\hat \go)\hat Z^c_a] K''(b-a)=0.
\end{equation}
Note that, since $\hat \bbE[g(\hat \go)]=1$, we can interpret $g(\hat \go)$ as a probability density.
It has the effect of tilting down the values of $\hat \go$ while keeping them independent.
We have, for any $i\le m$,
\begin{equation}
 { \hat\bbE}\left[  g(\hat\go) e^{{ \beta}\hat \go_i-\gl(\beta)}\right]= e^{-\zeta_m+ \gl(2\beta_2)-2\gl(\beta_2)} ,
\end{equation}
where 
$$ \zeta_m\coloneqq  \gl(2\beta_2)-\gl(2\beta-\gep_m)- 2\gl(\beta_2)+ \gl(\beta)+\gl(\beta-\gep_m) \stackrel{m\to \infty}{\sim}\gep_m (\gl'(2\beta_2)-\gl'(\beta_2))$$
 Recall that $\beta=\beta(m)= \beta_2+\frac{1}{m^2}$, and in particular that $|\beta-\beta_2|\ll \gep_m$ and hence $\beta$ can be replaced by $\beta_2$ when computing the asymptotic equivalence. Setting $\delta''_n=\ind_{\{n\in \tau''\}}$ and using Fubini as in \eqref{notsharp} as well as the definition of $K''$, we obtain
\begin{equation}\label{singlea}
 \hat \bbE\left[ g(\hat \go) \hat Z^c_a\right]= \bE\left[ e^{\sum_{k=1}^a \left(\gl(2\beta_2)-2\gl(\beta_2)-\zeta_m\right)  \delta_k} \delta_a\right]= \bE''\left[   e^{-\zeta_m \sum^a_{k=1}\delta''_k} \delta''_a\right].
\end{equation}
Considering a decomposition over the last renewal point in $[0,m-1]$ we deduce from \eqref{singlea} that 
\begin{equation}
 \sum_{a=0}^{m-1} \sum_{b=m}^{\infty} \hat \bbE\left[ g(\hat \go) \hat Z^c_a\right]= \bE''\left[   e^{-\zeta_m \sum^{m-1}_{k=1}\delta''_k} \right].
\end{equation}
Furthermore, we have
\begin{equation}
 \bE''\left[ e^{-\zeta_m\sum_{k=1}^{m-1} \ind_{\{k\in \tau''\}}}\right]\le \frac{1}{m}+
 \bP''\left(|\tau''\cap[1,m-1]|\le (\zeta_m)^{-1} \log m\right).
\end{equation}
Using the same truncation method as in \eqref{topz}, we have
$$\bP\left( \tau''_{k}\ge (m-1)\right)\le \frac{k\bE[\tau''_1\wedge (m-1)]}{m-1} \le k m^{-{ (\alpha\wedge 1)}+o(1)}$$
where the last inequality follows from  the definition of $\alpha$.
Replacing $k$ by $(\zeta_m)^{-1} \log m$ (which is $m^{1/2+o(1)}$) { and recalling that $\alpha>\frac 12$}, we conclude the proof of \eqref{lastwist}.
 \end{proof}

\section{Proof of Proposition \ref{forallwalks}}\label{apalf}

Recall that  $X^{(1)}$ and $X^{(2)}$ are two independent random walks with distribution $P$ (starting from the origin). We assume that $(X^{(1)}_n-X^{(2)}_n)_{n\ge 2}$ is transient (which is equivalent to $\beta_2>0$).
\subsection{Preliminaries}

We first prove the inequality \eqref{etaalphanu}. Let us denote by $N_n\coloneqq\sum_{k\geq n}\ind_{\{X^{(1)}_k=X^{(2)}_k\}}$ the number of collisions after time $n$. 
 The first thing we want to show is that (recall \eqref{defalpha})
\begin{equation}\label{compax}
  -\alpha=\limsup_{n\to \infty} \frac{\log  P^{\otimes 2}\left(N_{n}\geq 1\right)}{\log n}.
\end{equation}
The inequality  $\le$ is an immediate consequence of $P^{\otimes 2}\left(T\in[n,\infty)\right)
\leq P^{\otimes 2}(N_{n}\ge 1)$). For the other direction, observe that if  $N_{n(\log n)^2}\ge 1$ then either $N_1>(\log n)^2$ or there exists a gap larger than $n$ between two consecutive collisions, and thus  
by the strong Markov property,
\begin{equation}\label{NNN}
	P^{\otimes 2}\left(N_{n (\log n)^2}\geq 1\right)\leq P^{\otimes 2}\left(N_1>(\log n)^2\right)+\log(n)^2 P^{\otimes 2}\left(T\in \left[n,\infty\right)\right).
\end{equation}
Since $N_1$ is a geometric random variable, we obtain 
\begin{equation}
  P^{\otimes 2}\left(T\in \left[n,\infty\right)\right) \ge \frac{1}{(\log n)^2}\left(P^{\otimes 2}\left(N_{ n(\log n)^2}\geq 1\right)- P(T<\infty)^{(\log n)^2}\right).
\end{equation}
Since $P(T<\infty)<1$, this allows to conclude the proof of \eqref{compax}.
Next, we observe that by the strong Markov Property applied at the first intersection time after $n$,
\begin{equation}\label{yuiop}
P^{\otimes 2}\left(N_n\geq 1\right)=\frac{E^{\otimes 2}[N_n]}{E^{\otimes 2}[N_0]}=\frac{\sum_{k\geq n}P^{\otimes 2}\left(X^{(1)}_k= X^{(2)}_k  \right)}{E^{\otimes 2}[N_0]}.
\end{equation}
Since the denominator is simply a positive constant, it is easy to check that \eqref{yuiop} implies that
\begin{equation}
 1+ \varliminf_{k\to \infty}\frac{P^{\otimes 2}\left(X^{(1)}_k= X^{(2)}_k\right)}{\log k} \le  \varlimsup_{n\to \infty} \frac{\log  P^{\otimes 2}\left(N_{n}\geq 1\right)}{\log n} \le 1+ \varlimsup_{n\to \infty}\frac{P^{\otimes 2}\left(X^{(1)}_k= X^{(2)}_k\right)}{\log k}.
\end{equation}
Proving \eqref{etaalphanu} reduces then to proving that
\begin{equation}\label{deuxineqq}
 \liminf_{k\to \infty}\frac{P^{\otimes 2}\left(X^{(1)}_k= X^{(2)}_k\right)}{\log k}\ge -\frac{d}{\eta\wedge 2}   \quad  \text{ and }  \limsup_{k\to \infty}\frac{P^{\otimes 2}\left(X^{(1)}_k= X^{(2)}_k\right)}{\log k} \le -\nu.
\end{equation}
The second inequality in \eqref{deuxineqq} follows from the fact that $P^{\otimes 2}\left(X^{(1)}_k= X^{(2)}_k\right)= \|D^k\|^2_2\le \|D^k\|_{\infty}$.
To prove the first inequality in \eqref{deuxineqq} we first use Cauchy-Schwarz to observe that, for any $x\in \bbZ^d$,
\begin{equation}
 P^{\otimes 2}\left(X^{(1)}_k- X^{(2)}_k=x  \right)
 =\sum_{y\in \bbZ^d} D^{k}(y+x,0)D^{k}(y,0) \le \sum_{y\in \bbZ^d} D^{k}(y,0)^2=
  P^{\otimes 2}\left(X^{(1)}_k= X^{(2)}_k  \right).
\end{equation}
As a consequence we have, for every $L\ge 1$,
\begin{equation}\label{tippz}
   P^{\otimes 2}\left(X^{(1)}_k= X^{(2)}_k  \right)\ge \frac{1}{(2L+1)^d}  P^{\otimes 2} \left(|X^{(1)}_k- X^{(2)}_k|\le L\right).
\end{equation}
Taking $L=k^{\frac{1}{2\wedge \eta}+\frac \gep d}$ for $\gep>0$ arbitrary, we prove below that
\begin{equation}\label{toppz}
 \lim_{k\to \infty} P^{\otimes 2} \left(|X^{(1)}_k- X^{(2)}_k|\le k^{\frac{1}{2\wedge \eta}+\frac \gep d}\right)=1
\end{equation}
which implies that for $k$ sufficiently large
$P^{\otimes 2}\left(X^{(1)}_k= X^{(2)}_k  \right)\ge \frac{1}{2^{d+1}}k^{-\frac{d}{2\wedge \eta}-\gep}$ and hence, by sending $\gep\to 0$, that the first inequality in \eqref{deuxineqq} holds.

\medskip

We can restrict the proof of \eqref{toppz} to the case $d=1$. Moreover, the inequality follows from  Chebychev's inequality if $\eta>2$. If $\eta\le 2$, we truncate the increments of $X^{(1)}$ and $X^{(2)}$ at level $L= k^{\frac{1}{\eta}+\gep}$ before applying Chebychev's inequality: given $\delta>0$ and recalling \eqref{taileta}, this  yields for $k\ge k_0(\delta)$ sufficiently large,
\begin{equation}\begin{split}
  P^{\otimes 2} \left(|X^{(1)}_k- X^{(2)}_k|\ge L\right)
 &\le  2k P(|X_1|\ge L)+  L^{-2} 2k E\left[ X^2_1 \ind_{|X_1|\le L}\right]\\
 &\le   2k L^{-\eta+\delta}+ L^{-2}2k L^{2-\eta+\delta} \le 4k L^{-\eta+\delta}.
\end{split}\end{equation}
from which yields \eqref{toppz} by taking for instance $\delta=\eta^2\gep$. \qed

\subsection{Proof of Proposition \ref{forallwalks}}

The bound $\p(\beta)\ge 1+\frac{\eta\wedge 2}{2}$ is proved as
\cite[Corollary~2.20]{JL24_2} under the assumption that weak disorder holds, so we only have to prove the upper-bound $\p({ \beta_c})\leq 1+\frac{1}{\nu\vee 1}$. Moreover, if $\beta_2>0$ then by \cite[Corollary~2.11]{JL24_2} we have $\p(\beta_2)=2$, hence $\p(\beta_c)\leq 2$. Therefore we only have to prove $\p(\beta_c)\leq 1+\frac{1}{\nu}$ in the case $\nu>1$, which we assume in the following.
This claim is a consequence of the following result, which partly generalizes \cite[Theorem 1.2]{J23} to the case of an arbitrary reference random walk.
\begin{proposition}\label{glassy}
If $\p(\beta)\in ( 1+\frac{1}{\nu},2]$ then $\beta<\beta_c$ and $\lim_{u\to 0+} \p(\beta+u)=\p(\beta)$.
\end{proposition}

Given $T\ge 0$, we set $\cT_1:= \inf\{k\ge T \colon X^{(1)}_{k}=X^{(2)}_k \}$.
Given this value $T>0$, we introduce a new partition function by setting, for $t\ge T$ and $n\in\N$,
\begin{equation}\label{ladeff}
\mathcal Z^{\beta}_{n}(t,x):=E^{\otimes 2}\left[  \exp\left( \sum_{k=1}^{n}\left[\beta(\go_{k,X^{(1)}_k}+\go_{k,X^{(2)}_k})  - 2\gl(\beta)\right]\right)\ind_{\{ \cT_1=t \text{ and } X^{(1)}_{t}=x\} }\right],
\end{equation}
and $\mathcal Z^{\beta}_n(t,x)=0$ if $t\le T-1$.
Proposition \ref{glassy} follows from the combination of two technical result.
The first one is a bound on   $\bbE\left[  (W^{\beta}_{n})^p \right]$ which is uniform in $n$ and follows from a decomposition of 
the squared partition function  $(W^{\beta}_{n})^2$. We recall the definition $\chi(\beta):=e^{\gl(2\beta)-2\gl(\beta)}-1$.
\begin{lemma}\label{dfgh}
We have, for any $\beta>0$, $n\ge 0$ and $p\le 2$,
\begin{equation}
  \bbE\left[  (W^{\beta}_{n})^p \right]
  \le \bbE\left[ (W^{\beta}_{T-1})^p\right]\sum_{j\ge 0}\left((1+\chi(\beta))\sum_{t\ge T} \sum_{x\in \bbZ^d}\bbE\left[ \mathcal Z^{\beta}_{T-1}(t,x)^{p/2} \right]\right)^{j}.
\end{equation}
\end{lemma}
Of course, the above result is meaningful only if
$$ \sum_{t\ge T} \sum_{x\in \bbZ^d}\bbE\left[ \mathcal Z^{\beta}_{T-1}(t,x)^{p/2} \right]<\frac{1}{1+\chi(\beta)}.$$
As a consequence of our second technical result, this condition can be satisfied by taking $T$ sufficiently large, if $p$ lies in a certain range,
\begin{lemma} \label{dfghj} The following hold.
\begin{itemize}
 \item [(i)] If $p\in (1+\frac{1}{\nu}, \p(\beta))$ then we have
\begin{equation}
 \lim_{T \to \infty} \sum_{t\ge T} \sum_{x\in \bbZ^d} \bbE\left[ \mathcal Z^{\beta}_{T-1}(t,x)^{p/2} \right]=0.
\end{equation}
\item[(ii)]
If the value of $T$ is fixed and $p>1+\frac{1}{\nu}$, then
the function
$$\beta\mapsto \sum_{t\ge T} \sum_{x\in \bbZ^d} \bbE\left[ \mathcal Z^{\beta}_{T-1}(t,x)^{p/2} \right]$$
is continuous in $\beta$.
\end{itemize}

\end{lemma}

\begin{proof}[Proof of Proposition \ref{glassy}]
Assume $\p(\beta)\in(1+\frac 1\nu, 2]$.
 Since $\p(\beta)$ is nonincreasing in $\beta$, we need to show that for any ${p\in (1+\frac 1\nu,\p(\beta))}$ there exists $u>0$ such that $\p(\beta+u)>p$. First, using item $(i)$ of Lemma \ref{dfghj}, we choose $T$ in such a way that
 \begin{equation}
 \sum_{t\ge T} \sum_{x\in \bbZ^d} \bbE\left[ \mathcal Z^{\beta}_{T-1}(t,x)^{p/2} \right]\le \frac{1}{4(1+\chi(\beta+1))}.
 \end{equation}
 Then, using item $(ii)$ and the monotonicity of $\chi(\beta)$, we find $u\in(0,1)$ such that
 \begin{equation}
   \sum_{t\ge T} \sum_{x\in \bbZ^d} \bbE\left[ \mathcal Z^{\beta+u}_{T-1}(t,x)^{p/2} \right]\le \frac{1}{2(1+\chi(\beta+u))}.
 \end{equation}
 Finally, we use Lemma \eqref{dfgh} to conclude that
 \begin{equation}
  \sup_{n\ge 0}  \bbE\left[  (W^{\beta+u}_{n})^p \right] \le 2 \bbE\left[ (W^{\beta+u}_{T-1})^p\right]<\infty
 \end{equation}
which concludes the proof.
\end{proof}

\subsection{Proof of Lemma \ref{dfgh}}

We introduce, for $i\ge 1$,
\begin{equation}
 \cT_{i+1}:=\inf\left\{k\ge \cT_i+T \colon X^{(1)}_{k}=X^{(2)}_k \right\}.
\end{equation}
Given $n,j\ge 1$,  ${\bf t}:=(t_1,\dots,t_j)\in \bbN^{j} $ and ${\bf x}:=(x_1,\dots,x_j)\in (\bbZ^d)^j$, we define the event
$$A_{n,j}({\bf t},{\bf x})= \left\{ \forall i\in \lint 1,j\rint\colon  \cT_i=t_i  \text{ and } X^{(1)}_{t_i}=x_i  \right\} \cap \{ t_j\le n\}.$$
For $j=0$ we set
 $A_{n,0}= \{\cT_1>n\}$
and define, for $j\ge 0$,
\begin{equation}
 \mathcal Z^{\beta}_{n,j} ({\bf t},{\bf x}):=E^{\otimes 2}\left[ \exp\left( \sum_{k=1}^n\left[\beta(\go_{k,X^{(1)}_k}+\go_{k,X^{(2)}_k})  - 2\gl(\beta)\right]\right) \ind_{A_{n,j}({\bf t},{\bf x})} \right].
 \end{equation}
 Note that with this notation we have $\mathcal Z^{\beta}_{t}(t,x)=\mathcal Z^{\beta}_{t,1}(t,x)$.
 Since the events $(A_{n,j}({\bf t},{\bf x}))_{j,{\bf t},{\bf x}}$ partition the probability space, we have
\begin{equation}
 (W^{\beta}_n)^2=\sum_{j\ge 0} \sum_{({\bf t},{\bf x})\in \bbN^j\times (\bbZ^{d})^j} \mathcal Z^{\beta}_{n,j} ({\bf t},{\bf x}).
\end{equation}

Note now that $\mathcal Z^{\beta}_{n,j} ({\bf t},{\bf x})$ can be factorized using the Markov property for the random walks $X^{(1)}$ and $X^{(2)}$ at time $\cT_j$ and iterating.
Recalling the definition of \eqref{definishift}, for $j\ge 0$ and   $t_j\le n$ we obtain, setting $(t_0,x_0)=(0,0)$ as well as $\Delta_i{\bf t}=t_i-t_{i-1}$ and $\Delta_i{\bf x}=x_i-x_{i-1}$,
\begin{equation}
 \mathcal Z^{\beta}_{n,j} ({\bf t},{\bf x})
 =\left( \prod_{i=1}^{j}\theta_{t_{i-1},x_{i-1}} \mathcal Z^{\beta}_{\Delta_i{\bf t}}(\Delta_i{\bf t},\Delta_i{\bf x}) \right) \times \theta_{t_j,x_j} \mathcal Z^{\beta}_{n-t_j,0}.
 \end{equation}
Assuming that $p\le 2$ and combining subadditivity with shift invariance, we obtain that
\begin{equation}\label{umf}
\bbE\left[ (W^{\beta}_n)^p \right]\le \sum_{j\ge 0}  \sum_{({\bf t},{\bf x})}  \ind_{\{t_j\le n\}} \prod_{i=1}^{j}\bbE\left[ \mathcal Z^{\beta}_{\Delta_i{\bf t}}(\Delta_i{\bf t},\Delta_i{\bf x})^{p/2}\right] \times \bbE\left[ (\mathcal Z^{\beta}_{n-t_j,0})^{p/2} \right].
\end{equation}
For the last factor, we now observe that
\begin{equation}\label{umff}
 \bbE\left[ (\mathcal Z^{\beta}_{n,0})^{p/2} \right] \le \bbE\left[ \bbE\left[ \cZ^{\beta}_{n,0} \middle |  \cF_{T-1}\right]^{p/2} \right]\le \bbE\left[ (W^{\beta}_{T-1})^p\right]
\end{equation}
where for the last inequality we simply used Fubini as follows 
\begin{equation}\begin{split}
\bbE\left[ \cZ^{\beta}_{n,0} \middle |  \cF_{T-1}\right]&=E^{\otimes 2}\left[ \exp\left( \sum_{k=1}^{T-1}\left[\beta(\go_{k,X^{(1)}_k}+\go_{k,X^{(2)}_k})  - 2\gl(\beta)\right]\right)\ind_{\{\forall i\ge T\colon X^{(1)}_i\ne X^{(2)}_i\}} \right]\\
  &\le (W^{\beta}_{T-1})^2.
\end{split}\end{equation}
Similarly, for each of the factors in the product in \eqref{umf}, we estimate
\begin{align*}
\bbE\left[ \mathcal Z^{\beta}_{\Delta_i{\bf t}}(\Delta_i{\bf t},\Delta_i{\bf x})^{p/2}\right]&\leq \bbE\left[ \E\left[\mathcal Z^{\beta}_{\Delta_i{\bf t}}(\Delta_i{\bf t})\middle|\F_{T-1}\right]^{p/2}\right]\\
&=\bbE\left[ \mathcal Z^{\beta}_{T-1}(\Delta_i{\bf t},\Delta_i{\bf x})^{p/2}\right](1+\chi)^{p/2},
\end{align*}
 where the factor $1+\chi$ comes from the fact that $X^{(1)}$ and $X^{(2)}$ collide at time $\Delta_i{\bf t}$. Taking the sum over all ordered sequences of $(t_1,\dots,t_j)$ instead of restricting to $t_j\le n$, we obtain
\begin{equation}\label{umfff}
 \sum_{j\ge 0} \sum_{({\bf t},{\bf x})}   \prod_{i=1}^{j}\bbE\left[ \mathcal Z^{\beta}_{\Delta_i{\bf t}}(\Delta_i{\bf t},\Delta_i{\bf x})^{p/2}\right]
 \leq  \sum_{j\ge 0} \left((1+\chi)\sum_{t\ge T, x\in \bbZ^d}\bbE\left[ \mathcal Z^{\beta}_{T-1}(t,x)^{p/2}\right]\right)^j
\end{equation}
 and hence the desired result follows from the combination of \eqref{umf}--\eqref{umfff}.

\subsection{Proof of Lemma \ref{dfghj}}
Recall that $D$ is the transpose of the transition matrix of $X$ and the definition \eqref{defmuin} of $\mu_n$. We have
\begin{equation}\begin{split}\label{pti1}
 \mathcal Z^{\beta}_{T-1}(T-1+s,x)^{p/2} &=  (W^{\beta}_{T-1})^p P^{\otimes 2}_{\mu_{T-1}}\left( X^{(1)}_s=X^{(2)}_s=x, \forall r\in \lint 1,s-1\rint,  X^{(1)}_r \ne X^{(2)}_r \right)^{p/2}\\
&\le   (W^{\beta}_{T-1})^p \left( D^s\mu_{T-1}(x)\right)^p=   \left( D^s\hat W_{T-1}(x)\right)^p.
\end{split}\end{equation}
 Setting $Y_T:=  \sum_{s\ge 1}\sum_{x\in \bbZ^d}\left(D^s\mu_{T-1} (x)\right)^p$, the above yields
 \begin{equation}\label{dfgxx}
\sum_{t\ge T}\sum_{x\in \bbZ^d}  \mathcal Z^{\beta}_{T-1}(t,x)^{p/2}\le (W^{\beta}_{T-1})^p Y_T
\le \left( \sup_{n\ge 0} W_n \right)^p Y_T
 \end{equation}
 To prove item $(i)$ of the lemma, it is thus sufficient to show that
 \begin{equation}\label{shozup}
\lim_{T\to \infty} \bbE\left[ \sup_{n\ge 1}(W^{\beta}_{n})^p Y_T \right]=0.
\end{equation}
 We are going to show that that there exists some $C>0$ such that almost surely
\begin{equation}\label{ertyu}
 Y_T\le C \quad \text{ and } \lim_{T\to \infty}Y_T=0.
\end{equation}
 The convergence \eqref{shozup} follows from \eqref{ertyu} and dominated convergence for $\bbP$. To ensure domination, we use the fact that, by \cite[Theorem 2.1]{JL24_2}, $\sup_{n\ge 1}(W^{\beta}_{n})\in L^p$ for $p<\p(\beta)$. To prove \eqref{ertyu},
 we observe that
\begin{equation}\label{tyuiop}
 \sum_{x\in \bbZ^d}(D^s\mu_{T-1} (x))^p \le \min\left( \|\mu_{T-1}\|^p_p,  \|D^s\|^{p-1}_{\infty}  \right).
\end{equation}
The bound $\|\mu_{T-1}\|^p_p$ comes from the fact that convolution contracts $\ell^p$ norms and the other bound comes from the fact that
$(D^s\mu_{T-1}(x) )^p\le \|D^s\|^{p-1}_{\infty} D^s\mu_{T-1}(x).$
Since $p>1+\frac{1}{\nu}$, the definition of \eqref{defnu} implies that $\|D^s\|^{p-1}_{\infty}  $ is summable, which implies the first part of \eqref{ertyu}.

\medskip

We then note that weak disorder implies $\lim_{T\to \infty}\|\mu_{T-1}\|_p=0$ (this is for instance a consequence of \cite[Theorem 2.1]{CSY03}, which proves that $\lim_{n\to \infty} I_n=0$).
Hence, using the dominated convergence theorem for the sum over $s$ ($\|D^s\|^{p-1}_{\infty}$ is used for the domination), we  we obtain the second part of  \eqref{tyuiop}.

\medskip

Now let us prove item $(ii)$, which is an exercise in applying the dominated convergence theorem. For $\beta'<\beta''$ and every $k$ and $x$, we have the bound
\begin{equation}\label{dumbound}
(\beta\go_{k,x}-\gl(\beta))\le \beta'' (\go_{k,x})_+,
\end{equation}
valid on $[\beta',\beta'']$. For fixed $t\geq T-1$ and $x\in\bbZ^d$ and almost every $\omega$, applying dominated convergence to $P$ and using \eqref{dumbound} as the domination, we obtain that $\beta\mapsto \cZ_{T-1}^\beta(t,x)$ is continuous.   Next, we apply the dominated convergence theorem to the counting measure on $\bbN\times\bbZ^d$. From \eqref{pti1}   we have
\begin{equation}\label{triop}
 \sup_{\beta\in [\beta',\beta'']} \mathcal Z^{\beta}_{T-1}(t,x)^{p/2} \le  \sup_{\beta\in [\beta',\beta'']}\left(  D^{t-T}\hat W^{\beta}_{T-1}(x)\right)^p
\end{equation}
and we have, in the same way as \eqref{tyuiop},
\begin{equation}
\sum_{x\in \bbZ^d} \sup_{\beta\in [\beta',\beta'']} \left( D^{t-T} \hat W^{\beta}_{T-1}(x)\right)^p
\le   \sup_{\beta\in [\beta',\beta'']}\left(W^{\beta}_{T-1}\right)^p     \| D^{t-T} \|^{p-1}_{\infty}
\end{equation}
The first term is finite by \eqref{dumbound} and the second is summable in $t$. This implies that  $\beta\mapsto \sum_{t\ge T}\sum_{x\in \bbZ^d} \mathcal Z^{\beta}_{T-1}(t,x)^{p/2}$ is continuous for almost every $\omega$.
Finally, using for the combination of the above inequalities, we have
$$\sup_{\beta \in [\beta',\beta'']} \sum_{x\in \bbZ^d}\sum_{t\ge T} \mathcal Z^{\beta}_{T-1}(t,x)^{p/2} \le \sup_{\beta\in [\beta',\beta'']}\left(W^{\beta}_{T-1}\right)^p  \sum_{t\ge T}       \| D^{t-T} \|^{p-1}_{\infty}. $$
Since the r.h.s.\ is in $L^1$ (again by \eqref{dumbound}), we  applying dominated convergence with respect to $\P$
to obtain the claim of item $(ii)$.
\qed

 \section{Proof of Proposition~\ref{zeroinfinite}}\label{proofzero}

Before starting the proof let us expose its mechanism. Our choice for the distribution of $X_1$ was made so that it has a \textit{varying tail exponent}, namely
\begin{equation}\label{zeroquatre}
 \liminf_{n\to \infty}  \frac{\log P ( |X_1|>n)}{ \log n}=-4 \quad \text{ and } \quad   \limsup_{n\to \infty} \frac{\log P ( |X_1|>n)}{ \log n}=0.
\end{equation}
The statement about the $\liminf$ ensures that $(X_k)_{k\ge 0}$ is recurrent, which will be proved in Section~\ref{rec}, and thus by \eqref{lequiv} this implies that $\beta_2=0$
(the fact that $\beta_c=0$ is then implied by Proposition~\ref{strongbeta}). On the other hand the information about the $\limsup$ allow to replicate the argument given proof given in \cite[Theorem~1.1]{V23}, which proves that $\f(\beta)=0$ for all $\beta$ if the tail distribution of $X_1$ is sufficiently fat. This is the content of Section~\ref{free}

\subsection{Proving the recurrence of  $(X_k)_{k\ge 0}$}\label{rec}
Observe that $g(x)$ is symmetric and has a unique local maximum at zero, and thus  $P(X_k=x)= g^{\ast k}(x)$ has the same property for every $k$. This implies that
$$\max_{x\in \bbZ^d}P(X_k=x)=P(X_k=0) \quad \text{ and } \quad P(X_k=0)\ge P(X_{k+1}=0).$$
Using these properties, we have for any choice of $N$ and $M$
\begin{equation}
 \sum_{i=N+1}^{2N} P(X_i=0)\ge N P(X_{2N}=0)\ge \frac{N}{2N+1}P(|X_{2N}|\le N).
\end{equation}
In view of the above, to prove recurrence of $X$, it is sufficient to show that
\begin{equation}\label{limsup1}
 \limsup_{N\to \infty}P(|X_{2N}|\le N)>0,
\end{equation}
which we will do by showing that the $\limsup$ is equal to one. In order to estimate $P(|X_{2N}|\le N)$, we truncate the increments and apply the second moment methods.
Given $M\ge 1$, we set
\begin{equation}
 \bar X_N\coloneqq \sum_{i=1}^N (X_i-X_{i-1})\ind_{\{|X_i-X_{i-1}|\le M\}}.
\end{equation}
Using the union bound and Chebychev (note that $\bbE[|\bar X^2_1|]\le M^2$)
\begin{equation}\begin{split}
 P(|X_{2N}|\le N)&\ge   P(|\bar X_{2N}|\le N ; X_{2N}=\bar X_{2N}) \\ &\ge 1- P(|\bar X_{2N}|> N)- 2N P( |X_1|> M)\\
 &\ge 1-\frac{2 M^2}{N}- 2N \sum_{|x|> M} g(x).
\end{split}\end{equation}
We take $N=M^3$ and $M=a_k$ for $k\ge 1$. In that case $2M^2/N= 2a_k^{-1}$, which converge to zero as $k\to \infty$.
As for the other term, it is bounded as follows.
\begin{equation}
N\sum_{|x|> a_k} f(x)\le N  \sum_{i\ge k} \frac{1}{a^4_i} \le 4 N a^{-4}_k\le 4 a^{-1}_k.
\end{equation}
Taking $k\to \infty$ this proves \eqref{limsup1} and concludes the proof that $(X_k)_{k\ge 0}$ is recurrent. \qed

\subsection{Proving that $\f(\beta)\equiv 0$}\label{free}
Now we replicate the proof of \cite[Theorem~1.1]{V23} (proved under a slightly different assumption).
Let us consider a large $N$ (whose value will be determined later) and set
\begin{equation}\begin{split}
  V_N(x)&\coloneqq  E\left[ e^{\sum_{k=0}^{N-1}(\beta\go_{k+1,X_k+x}-\gl(\beta))} \ind_{\{ \forall k\in [0,N-1]\colon |X_k-x|\le N^2 \} } \right],\\
   \bar{V}_N(x)&\coloneqq  \frac{V_N(x)}{P\left( \forall k\in [0,N-1]\colon |X_k|\le N^2\right)}= \frac{V_N(x)}{\bbE\left[ V_N(x)\right]}.
\end{split}\end{equation}
In words, $\bar V_N(x)$ is the partition function started  at time $1$ from $x$ whose underlying random walk is conditioned to remain within distance $N^2$ from the starting point. We have, for any $k\ge 1$,
\begin{equation}\label{trop}
 W^{\beta}_N= \sum_{x\in \bbZ}  g(x)e^{\beta \go_{1,x}-\gl(\beta)}\theta_{1,x} W^{\beta}_{N-1}
\ge \sum_{|x|\le a_k}g(x) V_N(x)\ge g(|a_k|)\sum_{|x|\le a_k}V_N(x).
\end{equation}
Setting $h(N)=P\left( \forall k\in [0,N-1]\colon |X_k|\le N^2\right)$, we can rewrite the sum above as
\begin{equation}\begin{split}
 g(|a_k|)\sum_{|x|\le a_k}V_N(x)&= (2a_k+1)g(|a_k|)h(N) \left( \frac{1}{2a_k+1} \sum_{|x|\le a_k} \bar V_N(x) \right)
\\
&\eqqcolon   (2a_k+1)g(|a_k|)h(N) U_{N,k}.
 \end{split}\end{equation}
Now we set $N=N_k\coloneqq  \lceil \sqrt{a_{k-1}}\rceil$. With this setup, we are going to show that
\begin{equation}\begin{split}\label{deterandrandom}
&\lim_{k\to \infty}\frac{1}{N_k} \log \left[ (2a_k+1)g(|a_k|)h(N_k) \right]=0,\\
&\lim_{k\to \infty} U_{N_k,k}=1,
 \end{split}\end{equation}
where the second convergence holds in probability. The combination of the two  limits in \eqref{deterandrandom}  ensures that the following holds in probability
\begin{equation}
 \lim_{k\to \infty}\frac{1}{N_k} \log \left( g(|a_k|)\sum_{|x|\le a_k}V_{N_k}(x)\right)=0
\end{equation}
and hence from \eqref{trop} we deduce that $\f(\beta)= 0$ (we obtain $\ge$ from \eqref{trop} but the other inequality is trivial).
To prove  the first line in \eqref{deterandrandom}, we simply observe that $(2a_k+1)g(|a_k|)$ is of order $a^{-4}_{k-1}$ (or $N^{-8}_k$) and that $h(N)\ge P(|X_1|\le N)^N$ does not decay exponentially fast in $N$.

 \medskip

For the second line, since by construction $\bbE\left[ U_{N_k,k}\right]=1$, we only need to show that the variance goes to zero.
Using the fact that 
$$\bbE[\bar V^2_N(x)]\le e^{N(\gl(2\beta)-2\gl(\beta))}$$
and that $V^2_N(x)$ and $V^2_N(y)$ are independent if $|x-y|> 2N^2$ we obtain
\begin{equation}\label{estivariance}
 \mathrm{Var }(U^{\beta}_{N,k})\le e^{N (\gl(2\beta)-2\gl(\beta))} \frac{4N^2+1}{2a_k+1},
\end{equation}
and the result follows from our choice for $N_k$.\qed

\section*{Acknowledgments}
 The authors are grateful to the anonymous referee for bringing many suggestion to improve the quality of the manuscript, to Jinjiong Yu and Ran Wei for finding an error -- and a fix for it -- in the proof of Lemma \ref{qv_unbounded} and to Quentin Berger for indicating the simpler proof for Lemma \ref{babac}.
S.J. was supported by JSPS Grant-in-Aid for Scientific Research 23K12984. H.L.\ acknowledges the support of a productivity grand from CNQq and of a CNE grant from FAPERj.

\bibliographystyle{alpha}
\bibliography{ref.bib}

\end{document}